\documentclass[10pt]{amsart}
\usepackage{palatino, amssymb, amsfonts, latexsym, mathrsfs}
\usepackage{amssymb, amsmath, amscd}
\input xy
\xyoption{all}

\catcode`\@=11

\newlength{\tabwidth}
\newlength{\tabheight}
\newlength{\tabrule}
\newlength{\tabwidthx}
\newlength{\tabheightx}

\def\gentabbox#1#2#3#4{\vbox to \tabheight{\setlength{\tabrule}{#3}%
  \setlength{\tabwidthx}{#1\tabwidth}\addtolength{\tabwidthx}{\tabrule}%
  \setlength{\tabheightx}{#2\tabheight}\addtolength{\tabheightx}{-\tabheight}%
  \hbox to #1\tabwidth{%
    \hspace{-0.5\tabrule}\rule{\tabrule}{#2\tabheight}\hspace{-\tabrule}%
    \vbox to #2\tabheight{\hsize=\tabwidthx%
      \vspace{-0.5\tabrule}\hrule width\tabwidthx height\tabrule%
      \vspace{-0.5\tabrule}\vss%
      \hbox to \tabwidthx{\hss#4\hss}%
        \vss\vspace{-0.5\tabrule}%
      \hrule width\tabwidthx height\tabrule\vspace{-0.5\tabrule}}%
    \hspace{-\tabrule}\rule{\tabrule}{#2\tabheight}\hspace{-0.5\tabrule}}%
  \vspace{-\tabheightx}}}
\def\genblankbox#1#2{\vbox to \tabheight{\vfil\hbox to #1\tabwidth{\hfil}}}
\def\tabbox#1#2#3{\gentabbox{#1}{#2}{0.4pt}{\strut #3}}

\catcode`\:=13 \catcode`\.=13 \catcode`\;=13 \catcode`\>=13
\catcode`\^=13
\def:#1\\{\hbox{$#1$}}
\def.#1{\tabbox{1}{1}{$#1$}}
\def>#1{\tabbox{2}{1}{$#1$}}
\def^#1{\tabbox{1}{2}{$#1$}}
\def;{\genblankbox{1}{1}\relax}
\catcode`\:=12 \catcode`\.=12 \catcode`\;=12 \catcode`\>=12
\catcode`\^=7

\catcode`\@=11
\xyoption{all}

\newtheorem{theorem}{Theorem}[section]
\newtheorem{lemma}[theorem]{Lemma}
\newtheorem{prop}[theorem]{Proposition}
\newtheorem{cor}[theorem]{Corollary}

\theoremstyle{definition}
\newtheorem{definition}[theorem]{Definition}

\theoremstyle{remark}
\newtheorem{remark}[theorem]{Remark}
\newcommand{\U}{\dot{\mathbf U}}

\newcommand{\Ud}{\dot{\mathbf U}}
\newcommand{\ADA}{\mathfrak A_{D,\mathcal A}}
\newcommand{\tX}{\tilde{\mathcal X}_{\mathbf a,\nu}^\mathbf L}

\numberwithin{equation}{section}

\begin{document}

\title{On the geometric realization of the inner product and canonical basis for quantum affine $\mathfrak{sl}_n$}

\author{Kevin McGerty}
\address{Department of Mathematics, Imperial College London.}

\date{August 2010.}

\begin{abstract}
We give a geometric interpretation of the inner product on the modified quantum group of $\widehat{\mathfrak{sl}}_n$. We also give some applications of this interpretation, including a positivity result for the inner product, and a new geometric construction of the canonical basis.
\end{abstract}

\maketitle

\section{Introduction}
Let $\mathbf U$ be a quantized universal enveloping algebra.
The positive part $\mathbf U^+$ of $\mathbf U$ is well known to possess a canonical basis
\cite{K91}, \cite{L91}. In contrast, there is no particularly natural basis for the algebra $\mathbf U$ itself.  However Lusztig \cite{L92} has defined a variant of the quantized enveloping algebra known as the modified quantum group. This algebra has essentially the same representation theory, and can be given a canonical basis $\dot{\mathbf B}$ which packages together natural bases of the tensor product of a highest and lowest weight $\mathbf U$-module (when such modules exist), in the same way that the canonical basis $\mathbf B$ of $\mathbf U^+$ packages together natural bases of highest weight representations. Just as for $\mathbf B$, (c.f. \cite{GL93}, \cite{K91}) it is possible  to characterize this basis, up to sign, in terms of an involution and an inner product.

In \cite{BLM90} the quantized enveloping algebra of $\mathfrak{gl}_n$ was constructed geometrically as a limit of certain convolution algebras. Subsequently Lusztig \cite[Part 4, Notes 1]{L93}, \cite{L99}, and independently Ginzburg and Vasserot \cite{GV93}, observed that this construction could be extended to the case of quantum affine $\mathfrak{sl}_n$. More precisely, one can define a sequence of algebra $\mathfrak A_D$, and maps $\psi_D\colon \mathfrak A_D \to \mathfrak A_{D-n}$, along with compatible maps $\phi_D$ from the modified quantum group $\Ud(\widehat{\mathfrak{sl}}_n)$ such that the resulting map into the inverse limit of the system $(\mathfrak A_D,\psi_D)_{D \in \mathbb N}$ is injective (in fact we will give a new proof of this injectivity statement in Section \ref{injective}). One of the main results of the present paper is a construction of the inner product on the modified quantum group $\dot{\mathbf U}(\widehat{\mathfrak{sl}}_n)$ as a kind of limit of a natural family of inner products on the algebras $(\mathfrak A_D)_{D \in \mathbb N}$. We do this both in the context of function on $\mathbb F_q$-rational points and perverse sheaves, giving two proofs of the fact that our construction yields the inner product on the modified quantum group -- the first is elementary, using explicit formulae for multiplication in $\mathfrak A_{D,n,n}$, while the second, although requiring more machinery, gives a more conceptual explanation in terms of equivariant cohomology. As applications of these results we give a new construction of the canonical basis of $\Ud$ in this context (which was already known by \cite{ScV} using crystal basis techniques), and a prove a positivity property for the inner product of elements of $\dot{\mathbf B}$ which is conjectured to hold for arbitrary types.

\section{Background}
\label{background}
We begin by recalling the setup of \cite{L99}. Fix a positive integer $n$. Let $D$ be a positive integer, $\epsilon$ an indeterminate, $\mathbf k$ a finite field with $q$ elements and $v$ a square root of $q$. Given $V$ a free $\mathbf k[\epsilon,\epsilon^{-1}]$-module of rank $D$, a \textit{lattice} in $V$ is a free
$\mathbf k[\epsilon]$-submodule of $V$, of rank $D$. Let $\mathcal F^n$ denote the set of \textit{n-step periodic lattices} in $V$, that is, $\mathcal F^n$ consists of sequences of lattices $\mathbf
L = (L_i)_{i \in \mathbb Z}$ where $L_{i-1} \subset L_i$, and $L_{i-n}=\epsilon L_i$ for all $i \in \mathbb Z$. We will also write $\mathcal F^n_D$ if we wish to emphasize the rank of $V$. 
Throughout this paper, if $X$ is a finite set, we will write $|X|$ for the cardinality of $X$.

The group $G = \text{Aut}(V)$ of automorphisms of $V$ acts on $\mathcal F^n$ in the natural way. We shall be interested in functions supported on $\mathcal F^n$ and its square which are invariant with respect to the action of $G$ (where $G$ acts diagonally on $\mathcal F^n \times \mathcal F^n$). Thus we first
describe the orbits of $G$ on these spaces. Let $\mathfrak S_{D,n}$ be the finite set of all $\mathbf a=(a_i)_{i \in \mathbb Z}$ such that

\begin{itemize}
\item $a_i \in \mathbb N$;
\item $a_i=a_{i+n}$ for all $i \in \mathbb Z$;
\item for all $i\in \mathbb Z$, $a_i+a_{i+1}+\dotsb+a_{i+n-1}=D$.
\end{itemize}
For $\mathbf L \in \mathcal F^n$, let $|\mathbf L| \in \mathfrak S_{D,n}$ be given by $|\mathbf L|_i=\mathrm{dim}(L_i/L_{i-1})$. The $G$-orbits on $\mathcal F^n$ are indexed by this graded dimension: for $\mathbf a \in\mathfrak S_{D,n}$ set $\mathcal F_{\mathbf a}=\{\mathbf L \in \mathcal F^n\colon |\mathbf L|=\mathbf a\}$; then the $\mathcal F_{\mathbf a}$ are precisely the $G$-orbits on $\mathcal F^n$. The $G$ orbit on $\mathcal F^n \times \mathcal F^n$ are indexed, slightly more elaborately, by the set of matrices $\mathfrak S_{D,n,n}$, where $A=(a_{i,j})_{i,j \in \mathbb Z}$, is in $\mathfrak S_{D,n,n}$ if
\begin{itemize}
\item $a_{i,j} \in \mathbb N$; \item $a_{i,j}=a_{i+n,j+n}$ for all
$i,j \in \mathbb Z$; \item for any $i \in \mathbb Z$,
$a_{i,*}+a_{i+1,*}+\dotsb+a_{i+n-1,*}=D$; \item for any $j \in
\mathbb Z$, $a_{*,j}+a_{*,j+1}+\dotsb+a_{*,j+n-1}=D$.
\end{itemize}
Here
\[
a_{i,*}=\sum_{j\in \mathbb Z}a_{i,j}; \qquad a_{*,j}=\sum_{i\in
\mathbb Z}a_{i,j}.
\]
For $A\in \mathfrak S_{D,n,n}$ set
\[
r(A)=(a_{i,*})_{i\in \mathbb Z}\in \mathfrak S_{D,n} \qquad
c(A)=(a_{*,j})_{j \in \mathbb Z} \in \mathfrak S_{D,n}.
\]
For $A \in \mathfrak S_{D,n,n}$ the corresponding $G$-orbit $\mathcal O_A$ consists of pairs $(\mathbf L,\mathbf{L'})$ such that
\[
a_{i,j}=\mathrm{dim}\biggl(\frac{L_i\cap L_j'}{(L_{i-1}\cap
L_j')+(L_i\cap L_{j-1}')}\biggr),
\]
thus $\mathbf L \in \mathcal F_{r(A)}$ and $\mathbf{L'} \in \mathcal F_{c(A)}$.

Let $\mathfrak A_{D;q}$ be the space of integer-valued $G$-invariant functions on $\mathcal F^n \times \mathcal F^n$ supported on a finite number of orbits. If $e_A$ denotes the characteristic function of an
orbit $\mathcal O_A$, the set $\{e_A\colon A \in \mathfrak S_{D,n,n}\}$ is a basis of $\mathfrak A_{D;q}$. The space $\mathfrak A_{D;q}$ has a natural convolution product which gives it the structure of
an associative algebra. With respect to the basis of characteristic functions the structure constants are given as follows. For $A, B, C \in \mathfrak S_{D,n,n}$, let $\eta_{A,B;q}^C$ be the coefficient of $e_C$ in the product $e_Ae_B$. Then $\eta_{A,B;q}^C$ is zero unless $c(A)=r(B)$, $r(A)=r(C)$ and $c(B)=c(C)$.  Now suppose these conditions are satisfied and fix $(\mathbf L, \mathbf L'') \in \mathcal O_{C}$. Then $\eta_{A,B;q}^C$ is the number of points in the set 
\[
\{ \mathbf L' \in \mathcal F_{c(A))}\colon (\mathbf L, \mathbf L') \in \mathcal O_A, (\mathbf L',\mathbf L'') \in \mathcal O_B\}.
\]

Clearly this is independent of the choice of $(\mathbf L,\mathbf L'')$, and moreover it can be shown that these structure constants are polynomial in $q$, allowing us to construct an algebra $\mathfrak A_{D,\mathbb Z[t]}$ over $\mathbb Z[t]$, which is a free $\mathbb Z[t]$-module on a basis $\{e_A: A \in \mathfrak S_{D,n,n}\}$ such that $\mathfrak A_{D,\mathbb Z[t]|t=q} = \mathfrak A_{D,q}$. In fact we will use the algebra $\mathfrak A_{D,\mathcal A}$ which is obtained from $\mathfrak A_{D,\mathbb Z[t]}$ by extending scalars to $\mathcal A = \mathbb Z[v,v^{-1}]$ where $v^2=t$ and the $\mathbb Q(v)$-algebra $\mathfrak A_D$ obtained by extending scalars to $\mathbb Q(v)$ (we will, by deliberate misuse, treat $v$ as both an indeterminate and a square root of $q$, depending on the context). The algebra $\mathfrak A_D$ is sometimes known as the affine q-Schur
algebra. In what follows it will be more convenient to use a rescaled version of the basis $\{e_A\}_{A\in \mathfrak S_{D,n,n}}$
of $\ADA$, with elements $[A]=v^{-d_A}e_A$ where
\[
d_A = \sum_{i \geq k, j<l, 1\leq i \leq n}a_{ij}a_{kl}.
\]
Note that if we define $\Psi([A])=[A^t]$, where $(A^t)_{ij} = a_{ji}$, then it is straightforward to check that $\Psi$ is an algebra anti-automorphism (see \cite[Lemma 1.11]{L99} for details), which we will sometimes call the transpose anti-automorphism.

Next we introduce quantum groups. In order to do this we recall the notion of a root datum from \cite{L93}.

\begin{definition}
A \textit{Cartan datum} is a pair $(I,\cdot)$ consisting of a finite set $I$ and a $\mathbb Z$-valued symmetric bilinear pairing on the free Abelian group $\mathbb Z[I]$, such that
\begin{itemize}
\item $i\cdot i \in \{2,4,6,\ldots\}$
\item $2\frac{i\cdot j}{i\cdot i} \in \{0,-1,-2,\ldots\}$, for $i\neq j$.
\end{itemize}
A \textit{root datum} of type $(I,\cdot)$ is a pair $Y,X$ of finitely-generated free Abelian groups and a perfect pairing $\langle,\rangle \colon Y \times X \to \mathbb Z$, together with embeddings $I\subset X$, ($i\mapsto i$) and $I\subset Y$, ($i\mapsto i'$) such that $\langle i',j \rangle = 2\frac{i\cdot j}{i
\cdot i}$.
\end{definition}

Given a root datum, we may define an associated quantum group $\mathbf U$.
Since it is the only case we need, we will assume that our datum is symmetric
and simply laced so that $i\cdot i = 2$ for each $i \in I$, and $i\cdot j \in \{0,-1\}$
if $i \neq j$. In this case, $\mathbf U$ is generated as an algebra over $\mathbb Q(v)$
by symbols $E_i, F_i, K_\mu$, $i \in I$, $\mu \in Y$, subject to the following relations.
\begin{enumerate}
\item $K_0=1$, $K_{\mu_1}K_{\mu_2} = K_{\mu_1+\mu_2}$ for $\mu_1,\mu_2 \in Y$;
\item $K_{\mu} E_i K_{\mu}^{-1} = v^{\langle\mu,i'\rangle}E_i, \quad K_{\mu} F_i K_{\mu}^{-1} =
 v^{-\langle\mu,i'\rangle}F_i$ for all $i \in I$, $\mu \in Y$;
\item $E_iF_j - F_jE_i = \delta_{i,j}\frac{K_i-K_i^{-1}}{v-v^{-1}}$;
\item $E_iE_j=E_jE_i, \quad F_iF_j=F_jF_i$, for $i,j \in I$ with $i\cdot j =0$;
\item $E_i^2E_j +(v+v^{-1})E_iE_jE_i + E_jE_i^2 =0$ for $i,j \in I$ with $i\cdot j = -1$;
\item $F_i^2F_j +(v+v^{-1})F_iF_jF_i + F_jF_i^2 =0$ for $i,j \in I$ with $i\cdot j = -1$.
\end{enumerate}
Thus $\mathbf U$ is naturally $X$-graded, $\mathbf U = \bigoplus_{\nu \in X} \mathbf U_\nu$. The subalgebras $\mathbf U^+$ and $\mathbf U^-$ generated by the $E_i$s and $F_i$s respectively are isomorphic to each other, and indeed are isomorphic to the $\mathbb Q(v)$-algebra $\mathbf f$ generated by symbols $\{\theta_i: i \in I\}$ subject only to the relation
\[
\begin{split}
\theta_i\theta_j  - \theta_j\theta_i=0, & \text{ for } i, j \in I, \text{ with } i\cdot j =0;\\
\theta_i^2\theta_j +(v+v^{-1})\theta_i\theta_j\theta_i + \theta_j\theta_i^2 =0, &  \text{ for } i,j \in I \text{ with  } i\cdot j = -1.
\end{split}
\]
Note that the algebra $\mathbf f$ depends only on the Cartan datum. 

We also need to consider the modified quantum group $\dot{\mathbf U}$. This is defined by
\[
\U = \bigoplus_{\lambda \in X} \mathbf U 1_\lambda; \qquad \mathbf U 1_\lambda =
\mathbf U / \sum_{\mu \in Y} \mathbf U(K_\mu - v^{\langle \mu, \lambda \rangle}),
\]
where the multiplicative structure is given in the natural way by the formulae:
\[
1_\lambda x = x 1_{\lambda - \nu}, \quad x \in \mathbf U_\nu; \qquad 1_\lambda 1_{\lambda'} =
\delta_{\lambda, \lambda'}1_\lambda.
\]
Let $\Ud_\mathcal A$ be the $\mathcal A$-subalgebra of $\Ud$ generated by $\{E_i^{(n)}1_\lambda, F_i^{(n)}1_\lambda: n \in \mathbb Z_{\geq 0}, \lambda \in X\}$. It is known \cite{L93} that $\Ud_\mathcal A$ is an $\mathcal A$-form of $\Ud$, and moreover it is a free $\mathcal A$-module.

To describe the connection between our convolution algebra and quantum groups, we will need the following notation. For $\mathbf a \in \mathfrak S_{D,n}$ let $\mathbf{i_a} \in \mathfrak S_{D,n,n}$ be the diagonal matrix with $(\mathbf{i_a})_{i,j}=\delta_{i,j}a_i$. Let $E^{i,j}\in \mathfrak S_{1,n,n}$ be the matrix with $(E^{i,j})_{k,l}=1$ if $k=i+sn$, $l=j+sn$, some $s \in \mathbb Z$, and $0$ otherwise. Let $\mathfrak S^n$ be the set of all $\mathbf b=(b_i)_{i\in \mathbf Z}$ such that $b_i=b_{i+n}$ for all $i \in \mathbb Z$. Let $\mathfrak S^{n,n}$ denote the set of all matrices $A= (a_{i,j})$, $i,j \in \mathbb Z$, with entries in $\mathbb Z$ such that

\begin{itemize}
\item $a_{i,j} \geq 0$ for all $i \neq j$; \item
$a_{i,j}=a_{i+n,j+n}$, for all $i,j \in \mathbb Z$; \item For any
$i \in \mathbb Z$ the set $\{j \in \mathbb Z\colon a_{i,j}\neq 0\}$ is
finite; \item For any $j \in \mathbb Z$ the set $\{i \in \mathbb
Z\colon a_{i,j}\neq 0\}$ is finite.
\end{itemize}
\noindent
Thus we have $\mathfrak S_{D,n,n} \subset \mathfrak S^{n,n}$ for all $D$. For $i \in \mathbb Z/n\mathbb Z$, let $\mathbf i \in \mathfrak S^n$ be given by $\mathbf i_k=1$ if $k=i$ mod $n$, $\mathbf i_k=-1$ if $k=i+1$ mod $n$, and $\mathbf i_k=0$ otherwise. We write $\mathbf a \cup_i \mathbf{a'}$ if $\mathbf a=\mathbf{a'}+\mathbf i$. For such $\mathbf a,\mathbf{a'}$ set $_\mathbf a\mathbf e_\mathbf{a'} \in \mathfrak S^{n,n}$ to be $\mathbf{i_a}-E^{i,i}+E^{i,i+1}$, and $_\mathbf{a'}\mathbf f_\mathbf a \in \mathfrak S^{n,n}$ to be $\mathbf{i_{a'}}-E^{i+1,i+1}+E^{i+1,i}$. Note if $\mathbf a, \mathbf{a'} \in \mathfrak S_{D,n}$ then $_\mathbf a\mathbf e_\mathbf{a'}, _\mathbf{a'}\mathbf f_\mathbf a \in \mathfrak
S_{D,n,n}$. For $i \in \mathbb Z/n\mathbb Z$ set
\[
E_i(D)= \sum[_\mathbf a\mathbf e_{\mathbf{a'}}], \qquad  F_i(D)=
\sum[_\mathbf{a'}\mathbf f_{\mathbf a}],
\]
where the sum is taken over all $\mathbf a,\mathbf{a'}$ in $\mathfrak S_{D,n}$ such that $\mathbf a\cup_i\mathbf{a'}$. For $\mathbf a \in \mathfrak S^n$ set
\[
K_\mathbf a(D) =\sum_{\mathbf b \in \mathfrak S_{D,n}}v^{\mathbf
a\cdot \mathbf b}[\mathbf{i_b}]
\]
where, for any $\mathbf a, \mathbf b \in \mathfrak S^n$, $\mathbf a\cdot \mathbf b=\sum_{i=1}^n a_ib_i \in \mathbb Z$. If we let $X=Y=\mathfrak S^n$, and $I= \mathbb Z/ n\mathbb Z$, with the embedding of $I \subset X=Y$ given by $i \mapsto \mathbf i$ and pairing as given above, we obtain a symmetric simply-laced root datum. We call the quantum group associated to it $\mathbf U(\widehat{\mathfrak{gl}}_n)$ (in fact this is the degenerate, or level zero, form of the affine quantum group associated to $\widehat{\mathfrak{gl}}_n$). It can be shown \cite{L99} that the elements $E_i(D), F_i(D), K_{\mathbf a}(D)$, generate a subalgebra $\mathbf U_D$ which is a quotient of the quantum group $\mathbf U(\widehat{\mathfrak{gl}}_n)$, via map the notation suggests. Since $E_i(D),F_i(D),K_\mathbf a((D)$ all lie in $\mathfrak A_{D,\mathcal A}$ we similarly have an $\mathcal A$-subalgebra $\mathbf U_{D,\mathcal A}$. Note that this gives the algebra $\mathfrak A_D$ the structure of a $\mathbf U(\widehat{\mathfrak{gl}}_n)$-module.

Similarly, letting $\mathbf b_0 = (\ldots, 1,1,\ldots)\in \mathfrak S^n$, we have a root datum given by $X' = \mathfrak S^n/\mathbb Z\mathbf b_0$, $Y' = \{\mathbf a \in \mathfrak S^n: \mathbf a \cdot \mathbf b_0 = 0\}$ with embeddings $I \subset X', Y'$ induced by the above embeddings into $X$ and $Y$.  This new root datum is associated to (again the degenerate form of) the quantum group $\mathbf U(\widehat{\mathfrak{sl}}_n)$. We have an algebra map $\phi_D \colon \Ud(\widehat{\mathfrak{sl}}_n) \to \mathfrak A_D$ given by $E_i^{(n)}1_\lambda \mapsto E_i(D)^{(n)}[\mathbf{i_a}]$ where $\mathbf a \in \mathfrak S_{D,n}$ satisfies $\mathbf a \equiv \lambda  \text{ mod } \mathbb Z\mathbf b_0$ if such an $\mathbf a$ exists, and $E_i^{(n)}1_\lambda \mapsto 0$ otherwise, and similarly for the $F_i^{(n)}1_\lambda$. Clearly $\phi_D$ restricts to a map between the integral forms $\Ud(\widehat{\mathfrak{sl}}_n)_\mathcal A$ and $\ADA$. It can be readily checked, using a Vandermonde determinant argument, that the image of $\phi_D$ is exactly $\mathbf U_D$ (see \cite[Lemma 2.8]{L99}).

\section{Inner product on $\mathbf U_D$} \label{innerfn}
\begin{definition} We define a bilinear form
\[
(,)_D\colon \mathfrak A_{D;q} \times \mathfrak A_{D;q} \to \bar{\mathbb
Q_{l}}
\]
by
\begin{equation}
\label{ipdefinition}
(f,\tilde{f})_D = \sum_{\mathbf L,\mathbf{L'}}v^{\sum|\mathbf L|_{i}^{2}-\sum|\mathbf{L'}|_{i}^{2}}f(\mathbf L,\mathbf{L'})\tilde{f}(\mathbf L,\mathbf{L'}),
\end{equation}
for $f$ and $\tilde{f}$ in $\mathfrak{A}_{D,q}$, where $\mathbf L$
runs over $\mathcal{F}^n$ and $\mathbf{L'}$ runs over a set of
representatives for the $G$-orbits on $\mathcal{F}^n$.
\end{definition}

Let $\mathcal{O}_A$ be a G-orbit on $\mathcal{F}^n \times
\mathcal{F}^n$, and let
\[
 X_{A}^{\mathbf L} = \{\mathbf{L'} \in \mathcal F^n\colon (\mathbf L,\mathbf{L'}) \in \mathcal O_A\}.
\]
It is easy to check that
\begin{equation}
2d_A-2d_{A^t}=\sum_{i=1}^n a_{i,*}^2-\sum_{j=1}^n a_{*,j}^2.
\label{difference}
\end{equation}
Thus if $A,A'$ are in $\mathfrak S_{D,n,n}$ we find that
\[
(e_{A},e_{A'})_D = \delta_{A,A'}q^{d_A -
d_{A^t}}|X_{A^t}^{\mathbf{L'}}|,
\]
where $\mathbf{L'}$ is any lattice in $\mathcal F_{c(A)}$. Thus the basis $\{e_A: A \in \mathfrak S_{D,n,n}\}$ is orthogonal for our inner product, and hence $(,)$ is nondegenerate. If $\{\eta_{A,B;q}^{C}\}$ are the structure constants of $\mathfrak{A}_{D;q}$ with respect to the basis $\{[A]\colon A \in \mathfrak S_{D,n,n}\}$, then we have
\begin{equation}
([A],[A'])_D= \delta_{A,A'}v^{- d_{A^t}}\eta_{A^t,A;q}^{\mathbf
i_{c(A)}}.
\end{equation}
We therefore obtain an inner product on $\ADA$ taking values in $\mathcal A$ by defining
\begin{equation}
([A],[A'])_D= \delta_{A,A'}v^{- d_{A^t}}\eta_{A^t,A}^{\mathbf
i_{c(A)}} \in \mathbb Z[v,v^{-1}]
\end{equation}
By extending scalars, we obtain a $\mathbb Q(v)$-valued symmetric bilinear form on $\mathfrak A_D$.
We now give some basic properties of this inner product:

\begin{prop}
Let $A \in \mathfrak S_{D,n}$, and let $f, \tilde f \in \mathfrak
A_D$. Then we have
\[
([A]f, \tilde{f})_D = v^{d_A-d_{A^t}}(f,[A^t]\tilde f)
\]
\label{property}
\end{prop}

\begin{proof}
Clearly it suffices to establish this equation in the algebra $\mathfrak A_{D;q}$. Since the characteristic functions of G-orbits form a basis of $\mathfrak A_{D;q}$, we may assume that $f = e_B$ and $\tilde f = e_C$, moreover we may assume that 
\begin{equation}
r(A)=r(C), \quad c(A)=r(B), \quad c(B)=c(C). \label{assumption}
\end{equation}
as both sides are zero otherwise. It follows immediately that
\[
[A]\cdot e_B=v^{-d_A}e_A\cdot e_B , \qquad
v^{d_A-d_{A^t}}[A^t]\cdot e_C=v^{d_A-2d_{A^t}}e_{A^t}\cdot e_C.
\]
Hence if $(\tilde{\mathbf L},\mathbf {L'}) \in \mathcal O_C$ is fixed,
\begin{equation}
\begin{split}
([A]\cdot e_B,e_C)_D &=
q^{d_C-d_{C^t}}|X_{C^t}^{\mathbf{L'}}|\cdot v^{-d_A}
    |\{\mathbf{L''}\colon(\tilde{\mathbf L},\mathbf{L''}) \in \mathcal O_A,(\mathbf{L''},
    \mathbf{L'}) \in \mathcal O_B\}| \\
&=  v^\alpha|\{\mathbf L,\mathbf{L''}\colon(\mathbf L,\mathbf{L''})
\in \mathcal O_A, (\mathbf{L''},\mathbf{L'}) \in \mathcal O_B,
(\mathbf L,\mathbf{L'}) \in \mathcal O_C\}|,
\label{equationthefirst}
\end{split}
\end{equation}
where $\alpha = 2d_C-2d_{C^t}-d_A$. Similarly, if $(\tilde{\mathbf L}'',\mathbf{L'}) \in \mathcal O_B$ is fixed
\begin{equation}
\begin{split}
v^{d_A-d_{A^t}}(e_B,[A^t]\cdot e_C)_D
&=  q^{d_B-d_{B^t}}|X_{B^t}^{\mathbf{L'}}|\cdot v^{d_A-2d_{A^t}}|\{\mathbf{L}\colon(\tilde{\mathbf L}'',\mathbf{L}) \in \mathcal O_{A^t},(\mathbf{L},\mathbf{L'}) \in \mathcal O_C\}| \\
&=  v^\beta|\{\mathbf L, \mathbf{L''}\colon(\mathbf{L''},\mathbf{L})
\in \mathcal O_{A^t}, (\mathbf{L''},\mathbf{L'}) \in \mathcal O_B,
(\mathbf L,\mathbf{L'}) \in \mathcal O_C\}|,
\label{equationthesecond}
\end{split}
\end{equation}
where $\beta=2d_B-2d_{B^t}+d_A-2d_{A^t}$. But now considering the diagrams:

\xymatrix{ & & & & \mathbf L \ar[r]^A \ar[dr]_C & \mathbf{L''}
\ar[d]^B  & &
\mathbf L \ar[dr]_C & \mathbf{L''} \ar[l]_{A^t} \ar[d]^B\\
& & & & &\mathbf{L'} & & & \mathbf{L'} }

it is clear that the last line of equation (\ref{equationthefirst}) is the same as the last line of equation (\ref{equationthesecond}) if $\alpha=\beta$, that is, if 
\begin{equation}
2d_C-2d_{C^t}-d_A = 2d_B-2d_{B^t}+d_A-2d_{A^t}   \label{easy}
\end{equation}
But this follows directly from equation (\ref{difference}) and equation (\ref{assumption}).
\end{proof}

We have the following easy consequence:
\begin{cor}
Let $i \in \mathbb Z$, and let $f, \tilde f \in \mathfrak A_D$ and
$\mathbf c \in \mathfrak S^n$. Then we have
\begin{enumerate}
\item $(E_i(f), \tilde{f})_D = (f,vK_{\mathbf i}F_i(\tilde f))_D$
\item $(F_i(f), \tilde{f})_D = (f,vK_{-\mathbf i}E_i(\tilde f))_D$
\item $(K_{\mathbf c}(f),\tilde{f})_D = (f, K_{\mathbf c}(\tilde
f))_D$
\end{enumerate}
\label{properties}
\end{cor}

\begin{proof}
We may assume that $f = e_A$ and $\tilde f = e_{B}$. The third equation can then be checked immediately from the formulas above. The second equation follows from the symmetry of the inner product and the other two, so it only remains to prove the first. We may assume that $r(A)=r(B)-\mathbf i$ and $c(A)=c(B)$, as both sides are zero otherwise. Set $\mathbf a= r(A), \mathbf b= r(B)$. Now from the definitions we have
\[
E_i(e_A)=[_{\mathbf b}\mathbf e_{\mathbf a}]\cdot e_A , \qquad vK_{\mathbf i}F_i(e_B)=v^{1+\mathbf
i \cdot \mathbf a}[_{\mathbf a}\mathbf f_{\mathbf b}]\cdot e_B.
\]
Since $_{\mathbf b}\mathbf e_{\mathbf a} =$ $_{\mathbf a}\mathbf f_{\mathbf b}^t$, and $d_{_{\mathbf b}\mathbf e_{\mathbf a}}-d_{_{\mathbf a}\mathbf f_{\mathbf b}}=1+\mathbf i \cdot \mathbf a$ the result now follows immediately from the previous proposition.
\end{proof}

\begin{remark}
\label{transposeinnerproduct}
There is a unique algebra anti-automorphism $\rho\colon \mathbf U (\widehat{\mathfrak{gl}}_n) \to
\mathbf U(\widehat{\mathfrak{gl}}_n)$ such that
\[
\rho (E_i)=vK_{\mathbf i}F_i, \quad \rho(F_i)= vK_{-\mathbf i}E_i
\quad\rho(K_i)=K_{i}
\]
With this we may state the result of the previous corollary in the
form
\begin{equation}
\label{leftmodulecompatibility}
(u(f),\tilde f)_D=(f,\rho(u)\tilde f)_D,\qquad u\in \mathbf U(\widehat{\mathfrak{gl}}_n),\quad
f,\tilde f \in \mathfrak A_D.
\end{equation}

Note also that another natural choice\footnote{It is clear that one needs to restrict one factor to run over representatives of the $G$-orbits to avoid infinity sums -- in the finite type case considered in \cite{BLM90} this issue doesn't arise.} of definition for an inner product on $\mathfrak A_{D,q}$ would be given by taking the sum in Equation \ref{ipdefinition} over a set of representative of the $\text{Aut}(V)$-orbits on $\mathcal F$ in the first factor, and all lattices in the second (the opposite of our choice). This inner product which we denote by $(,)^t_D$ is obtained from the one we use via the antiautomorphism $\Psi$, that is 
\[
(f,\tilde{f})^t_D = (\Psi(f),\Psi(\tilde{f}))_D
\]
and thus obeys ``transposed'' versions of the properties established in this section so that 
\begin{equation}
\label{rightmodulecompatibility}
(fu,\tilde{f})^t_D = (f,\tilde{f}\bar{\rho}(u))^t_D, \qquad u \in \mathbf U(\widehat{\mathfrak{gl}}_n), \quad f,\tilde{f} \in \mathfrak A_D,
\end{equation}
where $\bar{\rho}$ is given by $\bar{\rho}(x) = \overline{\rho(\bar{x})}$. The precise relation between $(,)_D$ and $(,)^t_D$ can be given as follows: if $f, \tilde{f} \in [\mathbf i_\mathbf a]\mathfrak A_D[\mathbf i_\mathbf b]$ then 
\begin{equation}
\label{iprelation}
v^{\sum_{i=1}^n b_i^2}([[\mathbf b]]!)^{-1}(f,\tilde{f})_D = v^{\sum_{i=1}^n a_i^2}([[\mathbf a]]!)^{-1}(f,\tilde{f})^t_D.
\end{equation}
where we define $[[ \mathbf a]]!= \prod_{i=1}^n\prod_{j=1}^{a_i}(1-v^{-2j})$. In the finite type case of \cite{BLM90} this follows easily from considering the orbits on the product of the space of $n$-step flags with itself, while in the affine case it requires some more care. Since we will not use this fact we do not include the details. 

\end{remark}

\begin{lemma}
\begin{enumerate}
\item For $A \in \mathfrak S_{D,n,n}$, $([A],[A])_D \in
1+v^{-1}\mathbb Z[v^{-1}]$ \item For $A, A'\in \mathfrak
S_{D,n,n}$ and $A \neq A'$, $([A],[A'])_D=0$
\end{enumerate}
\label{values}
\end{lemma}
\begin{proof}
The second part of the statement is obvious. For the first, note that by \cite[4.3]{L99} the set $X^{\mathbf L'}_{A^t}$ is an irreducible variety of dimension $d_{A^t}$. Since we have
\[
([A],[A'])_D = \delta_{A,A'}q^{- d_{A^t}}|X_{A^t}^{\mathbf{L'}}|,
\]
the Lang-Weil estimates \cite{LW54} then show that $([A],[A])_D \in 1+v^{-1}\mathbb Z[v^{-1}]$, as required.
\end{proof}

\begin{remark}
The results of this section are analogues of the results of \cite[section 7]{L99}; however our inner product is \textit{not} the same as that of \cite[7.1]{L99}, and this make the proofs somewhat
simpler. Lusztig's choice of inner product has the property that it is compatible with the left and right $\mathbf U$-module structure, in the sense that Equations (\ref{leftmodulecompatibility}) and (\ref{rightmodulecompatibility}) both hold. Although this is not quite proved in \cite{L99} it follows from our above discussion using Equation (\ref{iprelation}). However, it does not produce the inner product on $\Ud$ in the way we need.
\end{remark}

\section{Inner product on $\dot{\mathbf U}$}
\label{innerU}
In this section we will write $\Ud$ to denote the modified quantum group $\Ud(\widehat{\mathfrak{sl}}_n)$ associated to the root datum $(X',Y')$ defined in Section \ref{background}. We wish to obtain an inner product on $\dot{\mathbf U}$ using those on the family of algebras $\{\mathbf U_D\}_{D \in \mathbb N}$.

We begin with some technical lemmas. Given $A \in \mathfrak S^{n,n}$ let $a_{i,\geq s} = \sum_{j \geq s}a_{i,j}$, and $a_{i, >s}$, $a_{i,\leq s}$, etc. similarly. 

\begin{lemma}
\label{multiplicationformulae}
a) Let $A \in \mathfrak S_{D,n,n}$ and $\mathbf{a'}=r(A)$. If there is an $\mathbf a \in \mathfrak S_{D,n}$ such that $\mathbf a \cup_i \mathbf a'$ (i.e. if $a_{i+1}'>0$) then we have
\begin{equation}
[_{\mathbf a}\mathbf e_{\mathbf a'}][A] = \sum_{s \in \mathbb Z,
a_{i+1,s}\geq 1} v^{a_{i,\geq s}-a_{i+1,>s}}
\left(\frac{1-v^{-2(a_{i,s}+1)}}{1-v^{-2}}\right)
[A+E^{i,s}-E^{i+1,s}],
\end{equation}
where $A=(a_{i,j})$.

b) Let $A' \in \mathfrak S_{D,n,n}$ and $\mathbf a=r(A')$. If there is an $\mathbf a' \in \mathfrak S_{D,n}$ such that $\mathbf a \cup_i \mathbf a'$ (i.e. if $a_i>0$) then we have
\begin{equation}
[_{\mathbf{a'}}\mathbf f_{\mathbf a}][A'] = \sum_{s \in \mathbb Z,
a'_{i,s}\geq 1} v^{a_{i+1,\leq s}'-a_{i,<s}'}
\left(\frac{1-v^{-2(a_{i+1,s}'+1)}}{1-v^{-2}}\right)
[A'-E^{i,s}+E^{i+1,s}],
\end{equation}
where $A'=(a'_{i,j})$.  \label{lemma4.1}
\end{lemma}
\begin{proof}
This follows by rescaling the statement of Proposition 3.5 in
\cite{L99}.
\end{proof}

Let $\mathcal R$ be the subring of $\mathbb Q(v)[u]$ generated by $\{v^j\colon j \in \mathbb Z\}$, and
\[
\prod_{i=1}^t(v^{-2(a-i)}u^{2}-1)/(v^{-2i}-1); \qquad a\in \mathbb Z, t \geq 1.
\]
For $A \in \mathfrak S^{n,n}$ let $_pA$ be the matrix with $(_pA)_{i,j}=a_{i,j}+p\delta_{i,j}$. We have the
following partial analogue of \cite[4.2]{BLM90}.

\begin{lemma}
Let $A_1,A_2, \ldots ,A_k$ be matrices of the form $_{\mathbf
a}\mathbf e_{\mathbf a'}$ or $_{\mathbf a}\mathbf f_{\mathbf a'}$,
for $\mathbf a,\mathbf{a'} \in \mathfrak S^n$, and $A$ any element of
$\mathfrak S^{n,n}$. Then there exist matrices $Z_1, Z_2,
\ldots,Z_m \in \mathfrak S^{n,n}$,  and $G_1, G_2,\ldots, G_m \in \mathcal R$ and an integer $p_0 \in \mathbb Z$ such that
\begin{equation}
[_pA_1][_pA_2]\ldots [_pA_k][_pA]= \sum_{i=1}^m
G_i(v,v^{-p})[_pZ_i], 
\end{equation}
for all $p \geq p_0$.
\end{lemma}
\begin{proof}
This follows using the same argument as in the proof of Proposition 4.2 in \cite{BLM90} (where a similar but stronger result is proved for the finite-type case). One uses induction on $k$. When $k=1$ the result follows from the previous lemma, once we note that both $a_{i,\geq s}-a_{i+1,>s}$
and $a_{i+1,\leq s}-a_{i,<s}$ are unchanged when $A$ is replaced with $A+pI$. 
\end{proof}

Recall from Section \ref{background} that there is a surjective homomorphism $\phi_D\colon \dot{\mathbf U} \to \mathbf U_D$ which, for $\lambda \in X$, sends $E_i 1_{\lambda} \mapsto E_i(D)[\mathbf i_{\mathbf a}]$ and $F_i 1_{\lambda} \mapsto F_i(D)[\mathbf i_{\mathbf a}]$  if there is an $\mathbf a$ in $\mathfrak S_{D,n}$ such that $\mathbf a = \lambda$ mod $\mathbb Z\mathbf b_0$, otherwise both $E_i 1_{\lambda}$, $F_i 1_{\lambda}$ are sent to zero. Let $\mathbf f$ be the algebra defined in Section \ref{background}. Pick a monomial basis of $\mathbf f$, $\{\zeta_i \colon i \in J\}$ say. The triangular decomposition for $\dot{\mathbf U}$ \cite[23.2.1]{L92} shows that $\mathfrak B=  \{\zeta_i^+\zeta_j^-1_{\lambda}\colon i,j \in J, \lambda \in X\}$ is a basis of $\dot{\mathbf U}$, where $+\colon \mathbf f \to \mathbf U^+$, and  $-\colon \mathbf f \to \mathbf U^-$ are the standard isomorphisms defined by $\theta_i \mapsto E_i$ and $\theta_i \mapsto F_i$ respectively. Define a bilinear pairing $\langle,\rangle_D$ on $\dot{\mathbf U}$ via $\phi_D$ as follows:
\[
\langle x,y\rangle_D = (\phi_D(x),\phi_D(y))_D
\]

\begin{prop}
Let $k \in \{0,1,\ldots,n-1\}$, then if $x,y \in \dot{\mathbf U}$
\[
\langle x,y \rangle_{k+pn}
\]
converges in $\mathbb Q((v^{-1}))$, as $p \to \infty$, to an
element of $\mathbb Q(v)$.
\end{prop}

\begin{proof}
We may assume that $x,y$ are elements of $\mathfrak B$. Then we need to show that
\[
\langle\zeta^+_{i_1}\zeta^-_{j_1}1_\lambda,\zeta^+_{i_2}\zeta^-_{j_2}1_\mu\rangle_{k+pn}
\qquad i_1,i_2,j_1,j_2 \in J; \lambda, \mu \in X
\]
converges as $p \to \infty$. Let $\iota\colon \mathbf f \to \mathbf f$ is the $\mathbb Q(v)$-algebra anti-automorphism fixing the generators $\theta_i, 1 \leq i \leq n$. Using Proposition~\ref{properties}, it is easy to see that this inner product differs from
\begin{equation}
\langle 1_\lambda,\iota(\zeta_{j_1})^+\iota(\zeta_{i_1})^-\zeta_{i_2}^+\zeta_{j_2}^-1_\mu\rangle_{k+pn}
\label{value}
\end{equation}
by a power of $v$ which is independent of $p$. But then the definition of the inner product and the previous proposition show that (\ref{value}) may be written as $G(v,v^{-p})$ for some $G \in \mathcal R$. The result then follows immediately from the definition of $\mathcal R$.
\end{proof}

\begin{remark}
\label{strongconvergence}
Note that the proof of the last proposition actually allows us to conclude that
\[
(\phi_D(\zeta^+_i\zeta^-_j1_\lambda),[_pA])_{k+pn}
\]
converges to an element of $\mathbb Q(v)$, as $p \to \infty$, for any $A \in \mathfrak S^{n,n}$. We will need this in the next section.
\end{remark}

\begin{definition}
We define
\[
\langle, \rangle \colon \dot{\mathbf U}\times \dot{\mathbf U} \to \mathbb Q(v),
\]
a symmetric bilinear form on $\dot{\mathbf U}$ given by
\[
\langle x, y\rangle = \sum_{k=0}^{n-1}\lim_{p \to \infty}\langle x,y\rangle_{k+pn}.
\]
\end{definition}

\begin{remark}
Although the inner products $(,)_D$ satisfy only Equation (\ref{leftmodulecompatibility}) and not Equation (\ref{rightmodulecompatibility}), the formula (\ref{iprelation}) which relates $(,)_D$ and $(,)_D^t$ can be used to show that our limiting inner product $\langle, \rangle$ satisfies the analogue of both equations, as indeed Lusztig shows for his inner product on $\Ud$ in \cite[Prop. 26.1.3]{L93}.
\end{remark}

\section{Comparison of inner products} 
\label{comparison}
Lusztig has shown that the algebra $\Ud$ has a natural inner product which characterised by the following result. (Again, in this section $\Ud$ denotes the modified quantum group attached to the root datum $(X',Y')$.)

\begin{theorem}[Lusztig]
\label{Udinnerproduct}
There exists a unique $\mathbb Q(v)$ bilinear pairing $(,) \colon \dot{\mathbf U}\times\dot{\mathbf U} \to \mathbb Q(v)$ such that
\begin{enumerate}
\item 
$(1_{\lambda_1}x1_{\lambda_2},1_{\mu_1}y1_{\mu_2}) =0 \qquad \forall x,y \in \dot{\mathbf U}$ unless
$\lambda_1=\mu_1,\lambda_2=\mu_2$; 
\item 
$( ux,y) = (x,\rho(u)y) \qquad \forall x,y \in \dot{\mathbf U}, u \in \mathbf U$; 
\item 
$( x^-1_\lambda,y^-1_\lambda)=(x,y), \qquad \forall x,y \in\mathbf f, \lambda \in X$.
 \end{enumerate}
Here $(x,y)$ is the standard inner product on $\mathbf f$, $($see
\cite[1.2.5]{L93}$)$. The resulting inner product is automatically
symmetric. \label{algebraic}
\end{theorem}
\begin{proof}
See \cite[26.1.2]{L93}.
\end{proof}

\begin{theorem}
The inner products $\langle,\rangle$ of section \ref{innerU} and $(,)$ of
Theorem~\ref{algebraic} coincide. \label{same}
\end{theorem}

The remainder of this section is devoted to a proof of this theorem. Property $(1)$ listed in Theorem~\ref{algebraic} clearly holds for $\langle,\rangle$, as the representatives for elements of $X$ in $\mathfrak S_{D,n}$ are distinct when they exist. Property $(2)$ follows from Corollary ~\ref{properties}; thus it only
remains to verify the third. 

Fix $\lambda \in X$. The algebra $\mathbf f$ is naturally graded: $\mathbf f = \bigoplus_{\nu \in \mathbb NI}\mathbf f_\nu$. For $\nu \in \mathbb Z[I]$, with $\nu = \sum_{i \in I}\nu_i i$ let $tr(\nu)= \sum_{i \in I}\nu_i$. If $z$ is homogeneous we set $|z| = \nu$, where $z \in \mathbf f_\nu$. Thus for the third property we may assume that $x,y \in \mathbf f$ are homogeneous, i.e. $x,y \in \mathbf f_\nu$ for some $\nu$, and proceed by induction on $N=\mathrm{tr}(\nu)$. If $N=0$ then we are reduced to the equation
\[
\langle 1_\lambda,1_\lambda\rangle= 1,
\]
which holds trivially. Now suppose that $N>0$ and the result is known for $x,y \in \mathbf f_\nu$ when $\mathrm{tr}(\nu)<N$. If $x,y$ are in $\mathbf f_\nu$, $\mathrm{tr}(\nu)=N$, then we may assume
that they are monomials, and $y=\theta_iz$ for some $z \in \mathbf f_{\nu - \mathbf i}.$ Thus we have
\[
\begin{split}
\langle x^-1_\lambda,y^-1_\lambda\rangle
&= \langle x^-1_\lambda,F_iz^-1_\lambda\rangle \\
&= \langle vK_{-\mathbf i}E_ix^-1_\lambda, z^-1_\lambda\rangle,
\end{split}
\]
using property $(2)$ of the inner product (which have already seen holds for both $(,)$ and $\langle,\rangle$). Now using standard commutation formulas (see \cite[3.1.6]{L93}) this becomes
\[
\langle vK_{-\mathbf i}x^-E_i1_\lambda, z^-1_\lambda\rangle+
\frac{1}{1-v^{-2}}\langle(_ir(x)^{-}-vK_{-\mathbf i}r_i(x)^-K_{-\mathbf i})1_\lambda, z^-1_\lambda\rangle
\]
where $_ir$ and $r_i$ are the twisted derivations defined in \cite[1.2.13]{L93}. Tidying this up we get
\[
\frac{1}{1-v^{-2}}\langle_ir(x)^-1_\lambda,z^-1_\lambda\rangle +\langle v^{\mathbf i
\cdot |x|-\mathbf i \cdot \lambda-1} \left(x^-E_i-\frac{v^{-\mathbf i
\cdot \lambda}}{v-v^{-1}}r_i(x)^-\right)1_\lambda,z^-1_\lambda\rangle
\]

But $_ir(x),z \in \mathbf f_{\nu-\mathbf i}$, hence by induction we have $\langle _ir(x)1_\lambda,z1_\lambda\rangle = (_ir(x)^-,z)$, and by standard properties of the inner product $(,)$ on $\mathbf f$ we know that $\frac{1}{1-v^{-2}}(_ir(x),z)= (x,\theta_iz)$, thus we are done by induction if we can show that for any $x\in \mathbf f_\nu$ the element 
\begin{equation}
\label{uelement}
u(x) = \left(x^-E_i-\frac{v^{-\mathbf i \cdot \lambda}}{v-v^{-1}}r_i(x)^-\right)1_\lambda
\end{equation}
annihilates $\mathbf U^-1_\lambda$. To see this we need some (rather technical) Lemmas about multiplication in $\mathfrak A_D$.

\begin{lemma}
Let $A \in \mathfrak S^{n,n}$ be such that $a_{r,s}=0$ for $r<s$ unless $r=s-1$ and $r=i$ mod $n$, when $a_{r,r+1}\in \{0,1\}$; then the following hold for $p$ sufficiently large.
\begin{enumerate}
\item For $j \neq i$ we have
\[
F_j[_pA]= \sum_{k=1}^mg_k(v)[_pZ_k]
\]
where $g_k(v) \in \mathcal A$ and $Z_k \in \mathfrak S^{n,n}$ ($1 \leq k \leq m)$ and moreover $g_k(v)$ is independent both of $p$ and $\{a_{r,s}\colon r \leq s\}$, and we have $(Z_k)_{r,s}=a_{r,s}$ for $r<s$. \item
\[
\begin{split}
F_i[_pA] &=\sum_{k=1}g_k(v)[_pZ_k] \\ &\qquad + v^{1-\mathbf
i\cdot r(A)} \biggl(\frac{1-v^{-2(a_{i+1,i+1}+1+p)}}{1-v^{-2}}
\biggr) [_p(A+E^{i+1,i+1}-E^{i,i+1})]
\end{split}
\]
where $g_k(v) \in \mathcal A$ is independent of $p$ and $\{a_{r,s}\colon r \leq s\}$, and we have $(Z_k)_{r,s}=a_{r,s}$ for $r<s$, and the final term occurs only if $a_{i,i+1}=1$.
\end{enumerate}
\label{mult}
\end{lemma}
\begin{proof}
By Lemma \ref{multiplicationformulae}, for any $A \in \mathfrak S^{n,n}$ and $p$ large enough we have,
\[
F_j[_pA]=\sum_{k\colon (_pA)_{j,k}\geq 1} v^{a_{j+1,\leq k}-a_{j,<k}}
\biggl( \frac{1-v^{-2(a_{j+1,k}+p\delta_{j+1,k}+1)}}{1-v^{-2}}
\biggr) [_pA+E^{j+1,k}-E^{j,k}].
\]
We claim that in our case the coefficients are independent of $p$ and of $\{a_{r,s}\colon r \leq s\}$, unless $j\equiv i \text{ mod } n$. Indeed then $(_pA)_{j,k} \geq 1$ implies that $j \geq k$, and hence the coefficient of $[_pA + E^{j+1,k}-E^{j,k}]$ in the above sum is
\[
v^{a_{j+1,\leq k}-a_{j,<k}}\big(\frac{1-v^{-2(a_{j+1,k}+1)}}{1-v^{-2}}\big)
\]
which evidently involves only entries $a_{r,s}$ of $A$ with $r>s$, thus establishing the first part of the Lemma.

For the second part, if $j \equiv i \text{ mod } n$ then we get the same conclusion except when $k=j+1$, if $a_{j,j+1}=1$ in which case $a_{j+1,\leq j+1} = a_{j+1,*}$ and $a_{j,<j+1} = a_{j,*}-1$ by our assumptions, so that the term $a_{j+1,\leq k} - a_{j, <k} = 1-\mathbf i\cdot r(A) $, and this yields the final term in second part, as required.
\end{proof}

Fix $\lambda \in X'$. Then $\sum_{i=1}^n \lambda_i =k \text{ mod } n$ for a well-defined $k \in \{0,1,\ldots n-1\}$. If $D = k+pn$, then there is a unique $\mathbf a \in X$ which is a representative of $\lambda \in X'$ satisfying $\sum_{i=1}^n a_i = D$ (that is, $\mathbf a \in \mathfrak S_{D,n}$). Let $\mathfrak A_D^- =\mathrm{span} \{[A]\colon a_{r,s}=0, \forall r<s\}$, and note that Lemma~\ref{mult} shows that $\phi_D(x^-1_\lambda) \in \mathfrak A_D^-$ for any $x \in \mathbf f$.

\begin{lemma}
\label{uglytechnicallemma}
Let $\sum_{j=1}^n\lambda_j = k$ mod $n$, where $k \in \{0,1,\dotsc,n-1\}$, and suppose that $D=k+pn$ for some $p$.  Then if $x$ is a monomial in the generators $\{\theta_i: i\in I\}$ we have, for sufficiently large $p$,
\begin{equation}
\label{Ecommutation}
\begin{split}
\phi_D(x^-E_i1_\lambda) = \sum_{k=1}^{m_1}a_k(v)[_pB_k] + \sum_{k=1}^{m_2}(b_k(v)+v^{-2p}c_k(v))[_pH_k]; \\
\phi_D(x^-1_\lambda) = \sum_{k=1}^{m_1}a_k(v)[_pB_k + E^{i+1,i+1} -E^{i,i+1}] 
\end{split}
\end{equation}
for some $B_k, H_k \in \mathfrak S^{n,n}$ independent of $p$, where $(B_k)_{r,s}=(H_k)_{r,s}= 0$ for $r<s$ unless $r =i \text{ mod }n$ when $(B_k)_{i,i+1}=1$, and the coefficients $a_k, b_k, c_k$ are independent of $p$, with $a_k \in \mathcal A$ and $b_k,c_k \in (v-v^{-1})^{-1}\mathcal A$.
Moreover, we have 
\[
\frac{v^{-\mathbf i\cdot \lambda}}{v-v^{-1}}\phi_D(r_i(x)^-1_\lambda) = \sum_{k=1}^{m_2}b_k(v)[_pH_k]
\]
\end{lemma}
\begin{proof}
We use induction on $N=\text{tr}(|x|)$. If $N=0$ then the result is clear, since we have $\phi_D(E_i1_\lambda) = [\mathbf i_\mathbf a+E^{i,i+1}-E^{i+1,i+1}]$ and $\phi_E(1_\lambda) = [\mathbf i_{\mathbf a}]$ (thus in this case we have $m_1=1$ and $m_2 =0$). Now suppose the result is known for all $y$ with $\text{tr}(|y|) <N$. We may write $x= \theta_jz$ where $\text{tr}(|z|) = N-1$. 

By induction, we have 
\[
\begin{split}
\phi_D(x E_i1_\lambda) &= F_j\phi_D(zE_i1_\lambda) \\
&= F_j\big(\sum_{k=1}^{m'_1}a'_k(v)[_pB'_k] + \sum_{k=1}^{m'_2}(b'_k(v)+v^{-2p}c'_k(v))[_pH'_k]\big)
\end{split}
\]
where $a'_k,b'_k,c'_k,B'_k,H'_k$ as in the statement of the Lemma (for $x=z$). Now applying Lemma \ref{mult} above to each of these terms, we find that
\begin{equation}
\label{inductionformula}
\begin{split}
\phi_D(xE_i1_\lambda) =& \sum_{k=1}^{m_1} a_k(v)[_pB_k] + \delta_{i,j}v^{1-\mathbf i\cdot r(B'_k)}\big(\frac{1-v^{-2((B'_k)_{i+1,i+1}+1+p)}}{1-v^{-2}}\big)a'_k(v)[_pB'_k+E^{i+1,i+1}-E^{i,i+1}] \\
&+ \sum_{k=1}^{m''_2}(b_k(v)+v^{-2p}c_k)[_pH_k]\\
\end{split}
\end{equation}
where $(B_k)_{r,s} =0$ if $r<s$ unless $r=i\text{ mod } n$ and $s=r+1$, and similarly $(H_k)_{r,s} =0$ if $r<s$. Thus the first formula of the Lemma is established by induction once we note that when $i\equiv j \text{ mod } n$ we may write
\[
\begin{split}
v^{1-\mathbf i\cdot r(B'_k)}& \big(\frac{1-v^{-2((B'_k)_{i+1,i+1}+1+p)}}{1-v^{-2}}\big)a'_k(v)[_pB'_k+E^{i+1,i+1}-E^{i,i+1}]\\
&=\frac{v^{-\mathbf i\cdot r(B'_k)}}{v-v^{-1}}\big(1-v^{-2p}v^{-2((B'_k)_{i+1,i+1}+1}\big)a'_k(v)[_pB'_k+E^{i+1,i+1}-E^{i,i+1}]\\
&= (b_{m''_2+k}+v^{-2p}c_{m''_2+k})[H_{m''_2+k}],
\end{split}
\]
where the last line defines $b_{m''_2+k},c_{m''_2+k}$ and $H_{m''_2+k}$, and by induction $b_{m''_2+k}, c_{m''_2+k} \in (v-v^{-1})^{-1}\mathcal A$, since $a'_k(v) \in \mathcal A$. Setting $m_2 = m''_2+\delta_{i,j}m_1$ the first formula is therefore established.

To show the second formula, we again use induction so that we have
\[
\phi_D(x^{-}1_\lambda) = F_j.\sum_{k=1}^{m'_1} a'_k(v)[_pB'_k+E^{i+1,i+1}-E^{i,i+1}] 
\]
Now by Lemma \ref{mult} (in particular the independence of the coefficients from the values of $\{a_{r,s}: r\leq s\}$ in all but the final term of the second formula) this is equal to
\[
 \sum_{k=1}^{m_1} a_k(v)[_pB_k+E^{i+1,i+1}-E^{i,i+1}], 
\]
as required. Finally, to see the moreover part of the Lemma, note that by definition we have
\[
\begin{split}
\biggl(\frac{v^{-\mathbf i \cdot \lambda}}{v-v^{-1}}\biggr)\phi_D(r_i(x)^-1_\lambda) 
&= \biggl(\frac{v^{-\mathbf i \cdot \lambda}}{v-v^{-1}}\biggr)\phi_D(r_i(\theta_j z)^-1_\lambda) \\
&= \biggl(\frac{v^{-\mathbf i \cdot \lambda}}{v-v^{-1}}\biggr)\big(\delta_{i,j}v^{\mathbf i\cdot |z|}\phi_D(z^-1_\lambda)+F_j(\phi_D(r_i(z)^-1_\lambda)\big)\\
\end{split}
\]
\noindent
Comparing this with Equation \ref{inductionformula}, we note that $r(B'_k)= \mathbf a +\mathbf i-|z|$, hence $1-\mathbf i \cdot r(B'_k)= \mathbf i \cdot (|z|- \lambda) -1$, so that
\[
\begin{split}
v^{1-\mathbf i \cdot r(B'_k)}\biggl(
\frac{1-v^{-2((B'_k)_{i+1,i+1}+p+1)}}{1-v^{-2}} \biggr ) =\biggl(
\frac{v^{\mathbf i \cdot (|z|- \lambda)}}{v-v^{-1}} \biggr)
(1-v^{-2p}v^{-2((B'_k)_{i+1,i+1}+1)}).
\end{split}
\]
hence the result follow once again by induction.
\end{proof}

Having established these technical Lemmas, it is now straight-forward to complete the proof of Theorem~\ref{same}. Let $\pi_D\colon \mathfrak A_D \to\mathfrak A_D^-$ be the orthogonal projection. Define $s_D\colon \mathbf f \to \mathfrak A_D^-$ by setting 
\[
x \mapsto \pi_D(\phi_D(x^-E_i1_\lambda))
\]
and define $r_D\colon \mathbf f \to \mathfrak A_D^-$ by setting
\[
x \mapsto \frac{v^{-\mathbf i \cdot \lambda}}{v-v^{-1}}\phi_D(r_i(x)^-1_\lambda).
\]
\noindent
The following is an immediate consequence of Lemma \ref{uglytechnicallemma}.
\begin{cor}
Let $x \in \mathbf f$.
\[
s_D(x)-r_D(x) = \frac{v^{-2p}}{v-v^{-1}}\biggl( \sum_{k=1}^mc_k(v)[_pZ_k] \biggr)
\]
for some $Z_k \in \mathfrak S^{n,n}$ and $c_k \in \mathcal A$, independent of $p$. Hence the element  $u(x)$ of Equation \ref{uelement} is orthogonal to $\mathbf U^-1_\lambda$.
\end{cor}
\begin{proof}
Let $y \in \mathbf f$ be a monomial. Then we have
\[
(u,y^-1_\lambda)= \lim_{p \to \infty} \langle
u,y^-1_\lambda \rangle_{k+pn},
\]
and by definition
\begin{equation}
\langle u,y^-1_\lambda \rangle_{k+pn}
=(s_{k+pn}(x)-r_{k+pn}(x),\phi_{k+pn}(y^-1_\lambda))_{k+pn}.
\label{usingprop}
\end{equation}
By the previous proposition,
\[
s_{k+pn}(x)-r_{k+pn}(x) = v^{-2p}\biggl( \sum_{j=1}^mc_j(v)[_pZ_j]
\biggr), \qquad Z_j \in \mathfrak S^{n,n},
\]
and by Remark \ref{strongconvergence}, we know that $([_pZ_j],\phi_{k+pn}(y^-1_\lambda))_{k+pn}$ converges in $\mathbb Q((v^{-1}))$ as $p \to \infty$. Thus the right-hand side of Equation~\ref{usingprop} tends to zero as required.
\end{proof}

\section{Geometric Interpretation} \label{geometric}

Recall from \cite[section 4]{L99} that $\mathfrak A_{D}$ possesses a canonical basis $\mathfrak B_D$ consisting of elements $\{A\}$, $A\in \mathfrak S_{D,n,n}$. To define these elements we must
assume $\mathbf k$ is algebraically closed (either the algebraic closure of $\mathbb F_q$, in which case we must use sheaves in the \'{e}tale topology, or $\mathbb C$ in which case we use the analytic
topology). Fix  $A \in \mathfrak S_{D,n}$, and $\mathbf L \in \mathcal F_{r(A)}$.

The space $\mathcal F^n$ can be given the structure of an ind-scheme such that each set $X_{A}^{\mathbf L}$ (see section \ref{innerfn})                                                                                                                                                                                                                                                                                                                                                                   lies naturally in a projective algebraic variety. This follows from the fact that if we fix $i_0,j_0 \in \mathbb Z$, then the subsets
\[
\mathcal F_{\mathbf b,\mathbf L}^p = \{\mathbf L' \in \mathcal F_{\mathbf b}: \varepsilon^pL_{i_0} \subset L'_{j_0} \subset \varepsilon^{-p}L_{i_0}\}
\]
(for $p=1,2,\ldots$) are naturally projective algebraic varieties each embedded in the next, and for any fixed $A \in \mathfrak S_{D,n,n}$ there is a $p_0 \in \mathbb Z$ such that $X^\mathbf L_A$ is a locally closed subset of $\mathcal F_{\mathbf b,\mathbf L}^p$ for all $p \geq p_0$. Thus its closure $\bar{X}_{A}^{\mathbf L}$ is naturally a projective algebraic variety. Let $\mathcal I^{\mathbf L}_A$ (or sometimes for convenience just $\mathcal I_A$) denote the simple perverse sheaf on $\bar{X}_{A}^{\mathbf L}$ whose restriction to $X_{A}^{\mathbf L}$ is $\mathbb C[d_A]$. Let $\mathcal H^s(\mathcal I^{\mathbf L}_A)$ be the $s$-th cohomology sheaf of $\mathcal I^{\mathbf L}_A$. For $A_1 \in \mathfrak S_{D,n,n}$ such that $X_{A_1}^{\mathbf L} \subset \bar{X}_{A}^{\mathbf L}$ we write $A_1 \leq A$, and set
\[
\Pi_{A_1,A} = \sum_{s \in \mathbb Z} \dim(\mathcal
H^{s-d_{A_1}}_y(\mathcal I^{\mathbf L}_A))v^s \in \mathbb Z[v^{-1}],
\]
where $\mathcal H^{s-d_{A_1}}_y(\mathcal I_A^{\mathbf L})$ is the stalk of $\mathcal H^{s-d_{A_1}}(\mathcal I_A^{\mathbf L})$ at a point $y \in \mathbf X_{A_1}^{\mathbf L}$ (since $\mathcal A^{\mathbf L}$ is constructible with respect to the stratification of $\bar{X}_A^{\mathbf L}$ given by $\{X_{A_1}^{\mathbf L}\colon A_1 < A\}$, this is independent of the choice of $y$). We let
\[
\{A\} = \sum_{A_1; A_1 \leq A}\Pi_{A_1,A}[A_1].
\]

Note that the following is an immediate consequence of the definitions and Lemma~\ref{values}. (It is the analogue for our inner product of \cite[Lemma 7.5]{L99}).
\begin{lemma}
\label{almostortho}
Let $A,A' \in \mathfrak S_{D,n,n}$, then,
\[
(\{A\},\{A'\})_D \in \delta_{A,A'} + v^{-1}\mathbb Z[v^{-1}].
\]
\label{values2} \hfill $\square$
\end{lemma}

The algebra $\mathfrak A_{D}$ may be viewed as a convolution algebra of (equivariant) complexes on $\mathcal F^n$. One must be slightly careful here since one cannot (at least straightforwardly) consider convolution on $\mathcal F^n \times \mathcal F^n$ as the ``complexes'' one would then have to consider would have infinite-dimensional support. However, \cite[4.2]{L99} gives one way in which this difficulty can be avoided: given $A, B \in \mathfrak S_{D,n,n}$ we may consider the set:
\[
Z = \{(\mathbf L',\mathbf L'') \in \mathcal F_\mathbf b \times \mathcal F_\mathbf c : \mathbf L' \in \bar{X}_{A}^\mathbf L, \mathbf L'' \in \bar{X}_{B}^{\mathbf L'}\}.
\]
As with $\bar{X}_A^\mathbf L$ we see that $Z$ is naturally a projective variety and the projection $\pi$ to the second factor gives a proper map $Z \to \mathcal F_{\mathbf c}$. The group $G_\mathbf L \subset \text{Aut}(V)$ of automorphisms stabilizing $\mathbf L$ acts on $\mathcal F_\mathbf b^p$ through a quotient which is naturally an algebraic group, and thus it makes sense to speak of $G_\mathbf L$-equivariant perverse sheaves on $\mathcal F_b$ and $Z$. If $\mathcal I$ denotes the middle extension of the constant sheaf on the smooth locus of $Z$, then $\mathcal I$ has a canonical $G_\mathbf L$-equivariant structure, and so by the decomposition theorem its push-forward along $\pi$ is a direct sum of (shifted) perverse sheaves of the form $\mathcal I_C$, ($C \in \mathfrak S_{D,n,n}$). We denote this push-forward by $\mathcal I_A *\mathcal I_B$. If $\mathcal K_{D,n,n}$ denotes the free $\mathcal A$-module on the set $[\mathcal I_A^\mathbf L]$ of isomorphism classes of the sheaves $\mathcal I_A^\mathbf L$ (as $\mathbf L$ runs over a set of $\text{Aut}(V)$-orbit representatives on $\mathcal F$) then the convolution $*$ gives $\mathcal K_{D,n,n}$ an associative $\mathcal A$-algebra structure, which is shown in \cite[4.4]{L99} to be isomorphic to $\mathfrak A_{D,\mathcal A}$ via the map $\Theta$ given by $\Theta([\mathcal I_A]) = \{A\}$. Moreover, Lusztig has shown that the submodule $K_{D,n,n}$ of $\mathcal K_{D,n,n}$spanned by the elements $\{[\mathcal I_A]: A \in \mathfrak S_{D,n,n}^\text{ap}\}$ is precisely the preimage of the subalgebra $\mathbf U_{D,\mathcal A}$, where $\mathfrak S_{D,n,n}^\text{ap}$ is the subset of $\mathfrak S_{D,n,n}$ consisting of those matrices $A \in \mathfrak S_{D,n,n}$ for which, given any $p \in \mathbb Z \backslash \{0\}$, there is a $k \in \mathbb Z$ with $a_{k,k+p} =0$.

\begin{remark}
\label{bartranspose}
Note also that this isomorphism yields the existence of an $\mathcal A$-anti-linear involution on $\mathfrak A_{D,\mathcal A}$ which fixes the basis elements $\{A\}$, ($A \in \mathfrak S_{D,n,n}$), by transporting via $\Theta$ the action of the Verdier duality functor.  We will write this involution as $x \mapsto \bar{x}$. Since it fixes the generators $E_i(D),F_i(D),K_\mathbf a$ it preserves the subalgebra $\mathbf U_{D,\mathcal A}$ and is compatible with the bar involution on $\Ud$ (see \cite[4.13]{L99} for more details). Moreover, Lusztig \cite[Proposition 4.12]{L99} shows that the anti-automorphism $\Psi$ corresponds to the map on $\mathcal K_{D,n,n}$ which sends $[\mathcal I_A]$ to $[\mathcal I_{A^t}]$.
\end{remark}

We wish to give an interpretation of the inner product of section \ref{innerfn} in the context of the algerba $\mathcal K_{D,n,n}$. Suppose that $A,B \in \mathfrak S_{D,n,n}$. We want to describe $(\{A\},\{B\})$. We may assume that $r(A)=r(B)=\mathbf a$ and $c(A)=c(B)=\mathbf b$. Let $\mathbf{L'} \in \mathcal F_{\mathbf b}$. Let $\mathcal I_{A^t}^{\mathbf L'}$ and $\mathcal I_{B^t}^{\mathbf L'}$ denote the simple perverse sheaves on $\bar{X}_{A^t}^{\mathbf{L'}}$ and $\bar{\mathbf X}_{B^t}^{\mathbf{L'}}$ respectively. Then define
\begin{equation}
\langle\mathcal I_A,\mathcal I_B\rangle^D = \sum_{i \in \mathbb Z}\dim(
H^i_c(\mathcal F_{\mathbf a},\mathcal I_{A^t}^{\mathbf L'}\otimes\mathcal I_{B^t}^{\mathbf L'}))v^i.
\label{perversedef}
\end{equation}
(here as usual $\otimes$ denotes the derived tensor product). Clearly $\langle,\rangle^D$ extends to an inner product on the whole of $\mathfrak A_{D}$ (viewed as an algebra of equivariant complexes on $\mathcal F^n$). We want to show that it is the same as the inner product $(,)_D$ of section \ref{innerfn} , at least on the subalgebra $\mathbf U_D$. We start by showing that $\langle,\rangle^D$ satisfies the properties of Proposition~\ref{properties}.

\begin{lemma}
Let $A, B, C \in \mathfrak S_{D,n,n}$, and suppose that
$\mathcal O_A$ is a closed orbit. Then
\[
\langle \mathcal I_A*\mathcal I_{B},\mathcal I_{C}\rangle^D=
v^{d_A-d_{A^t}}\langle\mathcal I_B,\mathcal I_{A^t}*\mathcal I_{C}\rangle^D.
\]
\label{geoad}
\end{lemma}

\begin{proof}
Both sides are obviously zero unless $r(A) = r(C)= \mathbf a$, $c(A)=r(B)=\mathbf b$ and $c(B)=c(C)=\mathbf c$, thus we assume these equalities from now on. Pick $\mathbf L^0 \in \mathcal F_{\mathbf c}$, and pick a subset $Y=\mathcal F_{\mathbf b,\mathbf L^0}^p$ of $\mathcal F_{\mathbf b}$ large enough so that such that $Y$ is a smooth projective variety containing $X_{B^t}^{\mathbf L_0}$. Let
\[
Z_A= \{(\mathbf L, \mathbf L') \in \mathbf\mathcal O_{A} \colon \mathbf L' \in Y\},
\]
We have maps $p_1\colon Z_A \to \mathcal F_{\mathbf a}$ and $p_2\colon Z \to Y$, the first and second projections respectively. The map $p_1$ is clearly proper (as the fibre is $X_A^{\mathbf L} \cap Y$) and the map $p_2$ is smooth with fibre dimension $d_{A^t}$. It follows that the complex $\mathcal I$ used in the definition of $\mathcal I_A*\mathcal I_B$ is the pull-back $p_2^*(\mathcal I_{B^t})[d_{A^t}]$, and hence we have
\[
(\mathcal I_A*\mathcal I_B)^t = (p_1)_!p_2^*(\mathcal I_{B^t})[d_{A^t}],
\]
and hence we have
\begin{equation}
\begin{split}
(\mathcal I_A*\mathcal I_B)^t\otimes \mathcal I_{C^t}&= (p_1)_!p_2^*(\mathcal I_{B^t})[d_{A^t}]\otimes \mathcal I_{C^t} \\
&= (p_1)_!(p_2^*(\mathcal I_{B^t})\otimes
p_1^*(\mathcal I_{C^t})[d_{A^t}]),
\end{split}
\end{equation}
where we use the projection formula in the second equality.

On the other hand, to compute the product $\mathcal I_{A^t}*\mathcal I_C$ we may similarly pick a smooth projective variety $W \subset
\mathcal F_{\mathbf a}$ which contains $X^{\mathbf L^0}_{C^t}$, and consider the variety
\[
Z_{A^t} = \{(\mathbf L, \mathbf L') \in \mathcal O_A \colon \mathbf L \in W\}.
\]
As above there are projection maps $p_1$,$p_2$, and the product is given by
\[
(\mathcal I_{A^t}*\mathcal I_{C})^t = (p_2)_!p_1^*(\mathcal I_{C^t})[d_{A}].
\]
so that
\[
\begin{split}
\mathcal I_{B^t}\otimes (\mathcal I_{A^t}*\mathcal I_{C})^t &= \mathcal I_{B^t} \otimes(p_2)_!p_1^*(\mathcal I_{C^t})[d_{A^t}] \\
&= (p_2)_!(p_2^*(\mathcal I_{B^t})\otimes p_1^*(\mathcal I_{C^t}))[d_{A}]
\end{split}
\]
where we again use the projection formula. Now since tensor product is local, we may restrict to $Z_A \cap Z_{A^t}$, and then it is clear that both inner products are given by the compactly supported cohomologies of $p_2^*(\mathcal I_{B^t})\otimes p_1^*(\mathcal I_{C^t})$ up to shift, with the difference in shifts being $d_A - d_{A^t}$ as required.

\end{proof}

\begin{lemma}
Let $A, B \in \mathfrak S_{D,n,n}$, and $\mathbf c \in \mathfrak
S^n$. Then
\begin{enumerate}
\item $\langle E_i\{A\},\{B\}\rangle^D= \langle\{A\},vK_{\mathbf
i}F_i\{B\}\rangle^D$.
\item $\langle F_i\{A\},\{B\}\rangle^D = \langle\{A\},vK_{-\mathbf i}E_i\{B\}\rangle^D$
\item $\langle K_\mathbf c\{A\},\{B\}\rangle^D=
\langle\{A\},K_\mathbf c\{B\}\rangle^D$
\end{enumerate}
\end{lemma}

\begin{proof}
This follows from the previous lemma exactly as in the proof of corollary \ref{properties}, since the varieties $X^{\mathbf L}_{_{\mathbf a + \mathbf i}\mathbf e_{\mathbf a}}$ are closed.
\end{proof}

The algebra $\mathbf U_D$ is spanned by elements of the form $T_1T_2\ldots T_N[\mathbf{i_a}]$ where $T_s$ is either $E_i$ or $F_i$ for some $i$. Thus in order to show that the inner products $(\,)_D$ and $\langle,\rangle^D$ coincide via the isomorphism the previous lemma shows we need only
check that
\[
\langle T_1T_2\ldots T_N[\mathbf{i_a}],[\mathbf{i_a}]\rangle^D =
(T_1T_2\ldots T_N[\mathbf{i_a}],[\mathbf{i_a}])_D
\]
But this will follow if we can show that
\[
\langle\{A\},[\mathbf{i_a}]\rangle^D = (\{A\},[\mathbf{i_a}])_D
\]
for all $A \in \mathfrak S_{D,n,n}$, as $\{\{A\}\colon A \in \mathfrak S_{D,n,n}\}$ is a basis of $\mathfrak A_D$. But as the simple perverse sheaf corresponding to $\{\mathbf{i_a}\}=[\mathbf{i_a}]$ is just the skyscraper sheaf at the point $\mathbf{L'}$, hence this last equality follow from directly from the definitions. We have therefore shown the following result.

\begin{prop}
On the algebra $\mathbf U_D$ the inner products $\langle,\rangle^D$ and $(,)_D$ coincide. \hfill $\square$ \label{same2}
\end{prop}

\begin{remark}
It can be shown that the algebra $\mathfrak A_D$ is generated by the elements $\{A\}$ for which $X_A^{\mathbf L}$ is closed, and so the above argument adapts to show that the inner products in fact agree on the whole of $\mathfrak A_D$, but we will not need this. Henceforth we will use the notation $(,)_D$ when referring to the inner product on $\mathfrak A_D$ in either of its incarnations.
\end{remark}

We now a second proof of the agreement of the limit of the inner products on $\mathbf U_D$ with Lusztig's inner product on $\Ud$. Recall that Theorem \ref{Udinnerproduct} characterises the inner product by three properties, the first two of which are evident for our limiting inner product. The difficulty then is establishing the third property, which relates the inner product on $\Ud$ to that on $\mathbf f$. We give a proof of this property, which while less elementary than the proof in Section \ref{comparison} is more conceptual. In the remainder of this section we will assume that our base field $\mathsf k$ is the field of complex numbers $\mathbb C$, and work with sheaves in the analytic topology. 

We first need to recall the geometric construction of the algebra $\mathbf f$ and its inner product (at least in the case of the cyclic quiver). Let $Q$ be the cyclic quiver $1 \to 2 \to \ldots \to n \to 1$. A representation of $Q$ is a $\mathbb Z/n\mathbb Z$-graded vector space $W= \bigoplus_{i \in \mathbb Z/n\mathbb Z} W_i$ equipped with linear maps $y_i \colon W_i \to W_{i+1}$ (where $i \in \mathbb Z/n\mathbb Z$). The space of such representations is denoted $E_W$. Such a representation is \textit{nilpotent} if there is an $N>0$ such that all compositions of $y_i$s of length greater than $N$ are equal to zero, and we write $E_W^{\text{nil}}$ for the subvariety of nilpotent representations. The group $G_W = \prod_{i \in \mathbb Z/n\mathbb Z} \text{GL}(V_i)$ acts on $E_W^{\text{nil}}$ with finitely many orbits. The algebra $\mathbf f$ associated to the Cartan datum of affine type $\widehat{\mathfrak{sl}}_n$ is then given as a convolution algebra of semisimple perverse sheaves which are $G_W$-equivariant and are supported on $E_V^{\text{nil}}$, thus since the stabilizer of a nilpotent representation is connected, the simple objects are labelled by the $G_W$-orbits on $E_W^{\text{nil}}$, that is, by the isomorphism classes of nilpotent representations. 

Given a pair $(t,m) \in \mathbb Z\times \mathbb Z_{\geq 0}$ we have a representation $V_{t,m}$ of the cyclic quiver with basis $\{e_j: t \leq j \leq t+p-1\}$ where $e_j$ has degree $j \text{ mod } n$, and $e_t \to e_{t+1} \to \ldots e_{t+p-1} \to 0$. The representations $V_{t,p}$ are a complete set of representatives for the isomorphism classes of indecomposable nilpotent representations of the cyclic quiver, thus since any nilpotent representation is a direct sum of indecomposables, we can record the isomorphism class of any such representation by a tableau $(\mu_{t,p})_{t,p \in \mathbb Z}$ where the entry $\mu_{t,p}$ records the multiplicity of $V_{t,m}$ in the representation.  We therefore have a natural parametrization of the canonical basis $\mathbf B$ by tableaux $(\mu_{t,p})_{t \in \mathbb Z,p \in \mathbb N}$, where $\mu_{t,p} \in \mathbb N$ and $\mu_{t,p} = \mu_{t-n,p}$ for all $t \in \mathbb Z$ and for fixed $t$ only finitely many of the $\mu_{t.p}$ are nonzero.

If $\dim(W_i) = \nu_i$, ($i \in \mathbb Z/n\mathbb Z$), the orbits of $G_W$ on $E_W^{\text{nil}}$ correspond to the isomorphism classes of nilpotent representations of dimension $\nu$ hence they are labelled by the set $\Sigma_\nu$ consisting of those tableaux $(\mu_{t,p})$ for which 
\[
\sum_{t,p; t \leq k < t+p} \mu_{t,p} = \nu_k.
\]
and so this same set indexes the $\nu$-homogeneous part of $\mathbf f$.

The inner product on $\mathbf f$ is defined in \cite[\S 12.2]{L93}. Let $A_1,A_2$ be $G_W$-equivariant simple perverse sheaves on $E_W^{\text{nil}}$. Note that the definition there is simply an explicit calculation of 
\begin{equation}
\label{finnerp}
(A_1,A_2) = \sum_{j \in\mathbb Z} \dim(H^{j}_{G,c}(A_1\otimes A_2))v^{-j}. 
\end{equation}
where $H^*_{G,c}$ denotes equivariant cohomology with compact supports (see below for more details). We thus briefly review some basics of the equivariant derived category. 

\begin{definition}
Let $X$ be a variety with a $G$-action (or more compactly, a $G$-variety). A \textit{resolution} of $X$ is a map $p \colon P \to X$ where $P$ is smooth $G$-variety on which $G$ acts freely (so that $\bar{P} = P/G$ is a smooth variety also). Let $\pi \colon P \to \bar{P}$ denote the quotient map. The category $D^b_G(X,P)$ consists of triples $(\mathcal F,\mathcal G, \phi)$ where $\mathcal F$ is an object in $D^b(\bar{P})$ and $\mathcal G$ is an object in $D^b(X)$ and $\phi\colon \pi^*(\mathcal F) \to p^*(\mathcal G)$ is an isomorphism. 
\end{definition}
\noindent
We will also need to recall the notion of an $n$-acylic map.
\begin{definition}
A map $f\colon Y \to X$ is said to be $n$-acyclic if
\begin{itemize}
\item
For any sheaf $F$ on $Y$ the adjunction morphism $B \to R^0f_*f^*(B)$ is an isomorphism, and $R^if_*f^*(F) = 0$ for $0<i\leq n$;
\item
For any base change $\tilde{X} \to X$ the induced map $\tilde{f}\colon \tilde{Y} = Y\times_X\tilde{X} \to \tilde{X}$ has the above property.
\end{itemize}
If we write $\tau_{\leq n}$ for the truncation functor on the derived category $D^b(Y)$, then the first condition may be re-written as saying that the adjuction map $F\to \tau_{\leq n}Rf_*f^*(F)$ in $D^b(Y)$ is an isomorphism for any sheaf $F$ (thought of as an complex in $D^b(Y)$ concentrated in degree $0$).
\end{definition}

For sufficiently acyclic resolutions $P$ (\textit{i.e.} resolutions $p\colon P \to X$ with $p$ an $n$-acyclic map for $n$ large), the cohomologies of objects in the category $D^b_G(X,P)$ can be used to calculate the cohomologies in $D_G^b(X)$ as indeed Bernstein and Lunts take a limit of resolution of $X$ to obtain their definition of the equivariant derived category. The construction in \cite[\S12.2]{L93} gives an explicit construction of a collection of $G$-resolutions of a variety which can be made arbitrarily highly connected (and hence his definition is the same as that of Equation \ref{finnerp}) however for our comparison result we need a more flexible context. 

To compare the inner products on $\mathbf f$ and $\mathbf U_D$ we need to recall Lusztig's construction \cite{L99} relating the geometry of periodic lattices to the cyclic quiver. Suppose that $\mathbf L$ is a fixed lattice, and $\mathbf a \in \mathfrak S^n$ is such that $\dim(V_i) = a_i$ (where on the left-hand side of this equality $i$ is understood to be taken modulo $n$). Consider the spaces:
\begin{itemize}
\item
$\mathcal X^{\mathbf L}_{\mathbf a,\nu} = \{\mathbf L' \in \mathcal F_{\mathbf a}: L'_i \subseteq L_i, \dim(L_i/L'_i)= \nu_i, \forall i \in \mathbb Z\}$
\item
$\tilde{\mathcal X}^\mathbf L_{\mathbf a, \nu} = \{(\mathbf L',(\phi_i)_{i \in \mathbb Z}: \mathbf L' \in \mathcal X^\mathbf L_{\mathbf a,\nu}, \phi_i \colon L_i/L'_i \to V_i \text{ an isomorphism}\}$
\item
$\mathcal U_\mathbf a \subset E_W^{\text{nil}}$ consists of those representations with label $(\mu_{t,p})$ such that 
\[
\mu_{t,1}+\mu_{t,2} + \ldots \leq a_t, \quad \forall t \in \mathbb Z.
\]
\end{itemize}

Both $\mathcal X^\mathbf L_{\mathbf a,\nu}$ and $\tilde{\mathcal X}^\mathbf L_{\mathbf a,\nu}$ can be given a natural structure of algebraic variety (with $\mathcal X^{\mathbf L}_{\mathbf a,\nu}$ projective), in the same fashion as for $\bar{X}_\mathbf L^A$ above, and the variety $\mathcal U_\mathbf a$ is an open subset of $E_W^{\text{nil}}$ (see \cite[Lemma 5.8]{L99}). We then have the following correspondence:

\xymatrix{ & & & & & &  \mathcal X^{\mathbf L}_{\mathbf a, \nu} & \ar[l]_{\alpha} \tilde{\mathcal X}^{\mathbf L}_{\mathbf a,\nu} \ar[r]^{\beta} & \mathcal U_{\mathbf a}
}
\noindent
where the map $\alpha$ is given by $(\mathbf L, \phi) \mapsto \mathbf L$, while the map $\beta$ is given by sending $(\mathbf L,\phi)$ to the element $(y_i) \in E_V^{\text{nil}}$ where $y_i$ given by the composition:

\xymatrix{& & & & &  V_i \ar[r]^{\phi_i^{-1}} & L_i/L'_i \ar[r] & L_{i+1}/L'_{i+1} \ar[r]^{\phi_{i+1}} & V_{i+1}}

\noindent
with the middle map induced by the inclusion $L_i \subseteq L_{i+1}$ (the point $(x_i)_{i \in \mathbb Z/n\mathbb Z}$ is automatically nilpotent as a representation of $Q$ by the periodicity of the flags $\mathbf L,\mathbf L'$). The map $\alpha$ is clearly a principal $G_W$-bundle, while the map $\beta$ is smooth with connected fibres of dimension $\sum_{1\leq i \leq n} a_i\nu_i$ (see \cite[Lemma 5.11]{L99}). 

Notice that if $\mathbf a$ has $a_i$ large enough for all $i$, then we have $E_W^{\text{nil}} = \mathcal U_\mathbf a$. In what follows we will always assume that this is the case. Moreover, the groups $G_\mathbf L$ (that is, the group of automorphisms of $V$ which preserve the lattice $\mathbf L$) and $G_W$ act naturally on $\tilde{\mathcal X}^{\mathbf L}_{\mathbf a,\mathbf b}$, making the maps $\alpha$ and $\beta$ equivariant (for the actions of $G_\mathbf L$ on $\mathcal X_{\mathbf a,\mathbf L}$ and $G_W$ on $E_W^{\text{nil}}$). Thus since $\tilde{\mathcal X}^{\mathbf L}_{\mathbf a,\nu}$ is free $G_W$-space (using the map $\alpha$) it is a resolution of $E_W^{\text{nil}}$.

Now Lusztig has shown in \cite[\S 5]{L99} that if $b$ is an element of the canonical basis which corresponds to the simple perverse sheaf $P$ on $E_W^{\text{nil}}$ with associated $G_W$-orbit corresponding to the tableau $(\mu_{t,p})$ then $\phi_D(b)[\mathbf i_\mathbf a]\in \mathfrak A_{D,n,n}$ is the element $\{B\}$ where $b_{i,i+j} = \mu_{i,j}$ and $b_{ii} = a_i - \sum_{p>0} \mu_{i,p}$, and $b_{ij} = 0$ if $i>j$. Moreover, we have
\[
\alpha^*(\mathcal I_B) \cong \beta^*(P)
\]
Picking an isomorphism $\theta$ (which is unique up to a scalar since $A$ is simple) we therefore obtain an element $\tilde{P} = (\alpha^*(\mathcal I_B), P, \theta)$ of $D^b_G(E_W^{\text{nil}},\tilde{\mathcal X}_{\mathbf a,\nu}^\mathbf L)$.
 Thus if $b_1,b_2$ are elements of $\mathbf B_\nu$ with associated tableau $(\mu_{t,p})$ and $(\rho_{t,p})$, and $P_1,P_2$ the corresponding perverse sheaves on $E_W^{\text{nil}}$, and $B_1$ and $B_2$ are the associated elements of $\mathfrak S^-_{D,n,n}$ then we may choose elements $\tilde P_k = (\mathcal I_{B_k},P_k,\theta_k)$ (where $k = 1,2$) in the category $D^b_G(E_W^{\text{nil}},\tilde{\mathcal X}_{\mathbf a,\nu}^{\mathbf L})$
\[
(\mathcal I_A,\mathcal I_B)_D = \sum_{j \in \mathbb Z} H^j_{G,\tilde{\mathcal X}^{\mathbf L}_{\mathbf a,\nu},c}(\tilde P_1,\tilde P_2)v^{-j}
\]

It follows that if we consider $\mathbf a' = \mathbf a + p\mathbf b_0$ instead of $\mathbf a$ for larger and larger $p$, this inner product will converge to $(b_1,b_2)$ provided the resolutions $\tilde{\mathcal X}_{\mathbf a',\nu}^\mathbf L$ become more and more highly connected as the $a'_i \to \infty$. Thus the compatibility of the inner products is reduced to showing that the maps $\beta \colon \tilde{X}_{\mathbf a,\nu}^\mathbf L \to E_W^{\text{nil}}$ is $k$-connected where $k \to \infty$ as $\text{min}\{a_i\} \to \infty$. The rest of this section will be devoted to a proof of this result. 

We wish to use a general Lemma which gives a criterion for a map to be $n$-acyclic. Since we cannot find a precise reference for what we need, we sketch briefly the result, though it is presumably well-known to the experts. The statement is a version of the Vietoris-Begle theorem proved in \cite[Prop. 2.7.8]{KS}. 

\begin{lemma}
\label{nacyclic}
Suppose we have a map $f\colon Y \to X$ which has $k$-connected fibres, and that we may exhaust $Y = \bigcup_{n}Y_n$, by closed subsets $Y_n$ such that $Y_n \subset \text{Int}(Y_{n+1})$ and the restriction of $f$ to $Y_n$ is proper with $k$-connected fibres for all $n$, then $\tau_{\leq k}Rf_*\circ f^*\cong \text{id}$. 
\end{lemma}
\begin{proof}
In fact the reference \cite[Prop 2.7.8]{KS} more is proved under the assumption that the fibres of $f$ are contractible, but the weaker statement that we need is precisely what follows from the proof given there. The key point is that in the case where $f$ is proper, one may use proper base change to conclude the vanishing of the functors $Rf_*^jf^*$ in the appropriate range from the $k$-connectedness of the fibres. The extension to the non-compact case then follows by via the Mittag-Leffler condition.
\end{proof}

Since the hypotheses of Lemma \ref{nacyclic} are preserved by base-change, it yields a criterion for a map to be $n$-acylic. We now use the above Lemma to show that $\beta$ is a $k$-acyclic map for $k= \mathrm{min}_{0\leq i \leq n-1}\{(a_i-\nu_i)\}$.

\begin{lemma}
\label{fibres}
The fibres of $\beta\colon \tilde{\mathcal X}_{\mathbf a,\nu}^\mathbf L \to \mathcal U_\mathbf a$ are $k$-connected for $k=2.\mathrm{min}_{1 \leq i \leq n}\{(a_i - \nu_i)\}$.
\end{lemma}
\begin{proof}
First note that we may view $X = \tX$ as the set $\{(\varphi_i)_{i \in \mathbb Z} : \varphi_i \colon L_i/\text{ker}(\varphi_{i-1}) \to W_i\}$ where $\varphi_i$ is surjective, $\text{ker}(\varphi_i)$ is a lattice, and $\varphi_{i-n} = \epsilon \varphi_i \epsilon^{-1}$. The corresponding pair $(\mathbf L',(\phi_i))$ is given by $\mathbf L' = (\text{ker}(\varphi_i))_{i \in \mathbb Z}$ with the isomorphisms $\phi_i \colon L_i/L'_i \to W_i$ induced by the surjections $\varphi_i$.

Now suppose that $y=(y_i)_{i \in \mathbb Z/n\mathbb Z} \in E_W^{\text{nil}}$ is a nilpotent representation of the cyclic quiver, and that $(\varphi_i)_{i \in \mathbb Z}$ is in the fibre of $y$. Considering the diagram:

\xymatrix{& & & & 0 \ar[r] & \text{ker}(\varphi_{i+1}) \ar[r] & L_{i+1} \ar[r]^{\varphi_{i+1}} & W_{i+1} \ar[r] & 0\\
& & & & 0 \ar[r] & \text{ker}(\varphi_{i}) \ar[r] \ar[u] & L_{i} \ar[r]^{\varphi_i} \ar[u] & W_{i+1} \ar[u]^{y_i} \ar[r] & 0
}
\noindent
we see that the restriction of $\varphi_{i+1}$ is determined on $L_i$ as it is given by $y_i \circ \varphi_i$ there. Thus given $\varphi_i$, the collection of $\varphi_{i+1}$s which induce $y_i$ is given by choosing a surjection $\pi\colon L_{i+1}/L_i \to W_{i+1}/\text{im}(y_i)$, and then picking a lift of the pair of maps $(\pi,y_i\circ \varphi_i)$ to a map $\varphi_{i+1}\colon L_i/\text{ker}(\varphi_i) \to W_{i+1}$ (since any such lift will be a surjective map). Thus the space of such choices is homotopy equivalent to the space of surjections from $L_{i+1}/L_i$ to $W_{i+1}/\text{im}(y_i)$. This is a complex Stiefel manifold, and hence $2(a_i-\nu_i+\text{rank}(y_i))$ connected.

Thus if we set $k = 2.\mathrm{min}_{1\leq i \leq n}\{(a_i - \nu_i)\}$ we may view $\beta^{-1}(y)$ as an iterated sequence of fibre bundles over the space $B$ of surjections $\{\varphi_0\colon L_0 \to W_0\}$ where in each case the fibres are at least $k$-connected. Thus by the standard long exact sequence $\beta^{-1}(y)$ will be $k$-connected provided we can show that $B$ is. Now $B$ is the space of surjective linear maps from $L_0$ to $W_0$ which intertwine the action of $\epsilon$ with the composition $\theta = y_ny_{n-1}y_{n-2}\ldots y_0\colon W_0 \to W_0$ (note that since the action of $\epsilon$ is nilpotent, this shows the representation $y$ must be also) and the following Lemma shows that in fact $B$ is at least $2(D-\nu_0)\geq 2(a_0-\nu_0)$ connected, so we are done.
\end{proof}

\begin{lemma}
\label{nilpfibre}
Let $(U,\theta)$ be a finite-dimensional vector space $U$ equipped with nilpotent endomorphism $\theta$ of Jordan type $\lambda$, and let $L$ be a free $\mathbb C[\epsilon]$-module of rank $D$. Viewing $U$ as a $\mathbb C[\epsilon]$-module via $\epsilon \mapsto \theta$, the space of $\mathbb C[\epsilon]$-module surjections $\varphi \colon L \to U$ is $d$-connected, where $d= 2(D-\ell(\lambda))$ and $\ell(\lambda)$ is the length of the partition $\lambda$.
\end{lemma}
\begin{proof}

Note that a $\mathsf k[\epsilon]$-module map $\varphi\colon L \to U$ is surjective if and only if the induced $\mathsf k$-linear map $\bar{\varphi} \colon L/\epsilon(L) \to U/\theta(U)$ is surjective. Moreover, since the space of surjections from $L/\epsilon(L)$ to $U/\theta(U)$ is a complex Stiefel manifold of $k$-frames in a $D$ dimensional space, it is $2(D-\ell(\lambda))$-connected. 

Now suppose that $\psi$ is a $\mathbb C$-linear map from $L/\epsilon(L) \to U/\theta(U)$. Let $K$ denotes the $D$-dimensional $\mathbb C$-vector space spanned by a set of $\mathbb C[\epsilon]$-generators of $L$, say $\{e_1,\ldots,e_D\}$. A $\mathsf k[\epsilon]$-module map $\varphi\colon L \to U$ with $\bar{\varphi}=\psi$ is completely determined by its restriction to $K$, and the induced map $\varphi'\colon K \to U$ is given by a choice of lifts for the vectors $\{\psi(e_i)\}$ in $U$ (where by abuse of notation we denote by $\psi$ the composition of $L \to L/\epsilon(L) \to U/\theta(U)$), and hence the space of such surjections is clearly a vector bundle over the Stiefel manifold of surjections from $L/\epsilon L$ to $U/\theta(U)$ which proves the Lemma.

\end{proof}

\begin{prop}
The map $\beta\colon \tX \to E_W^{\text{nil}}$ is $k$-connected.
\end{prop}
\begin{proof}
To show $\beta$ is a $k$-connected map we wish to apply Lemma \ref{nacyclic}. By Lemma \ref{fibres} we know that the fibres of $\beta$ are $k$-connected, and hence we must show that we can filter $\tX$ by subvarieties $\{Y_i\}_{i \in \mathbb N}$ such that $\beta_{|Y_i}$ is proper while ensuring that the fibres remain $k$-connected. To do this we simply note that the topology of the fibres of $\beta$ all come from Steifel manifolds, and these deformation retract on to the compact Steifel manifolds. Moreover the retraction can be done via the ``Gram-Schmidt'' process, once we endow our vector spaces with a Hermitian inner product. 

More precisely, we may equip $V$ with an Hermitian inner product (here we will assume that $\mathsf k= \mathbb C$, as in the rest of this section) so that $\epsilon$ is a unitary map (\textit{e.g.} take a $\mathbb C[\epsilon]$-basis $\{e_1,e_2,\ldots e_D\}$ of $L_0$ and define the Hermitian product $\langle,\rangle$ by setting 
\[
\langle \epsilon^le_j,\epsilon^m e_k \rangle = \delta_{j,k}\delta_{l,m}, \quad (l,m \in \mathbb Z, 1\leq j,k \leq D),
\]
Similarly we may equip the $\{W_i\}$ with Hermitian inner products. Then we have a norm function $\mathrm N$ on $\tX$ given by
\[
\mathrm N(\mathbf L,(\phi_i)) = \text{max}_{0 \leq i \leq n-1}\{\text{sup}\{\|\phi_i(u)\|:u \in L_i/L'_i, \|u\| = 1\}.
\]
where the norm on $L_i/L_i'$ is induced from that on $V$ via the canonical isomorphism $(L_i')^\perp \cong L_i/L_i'$. Now we may set $Y_i = \{(\mathbf L,(\phi_i)): i^{-1}\leq \mathrm N(\mathbf L,(\phi_i)) \leq i\}$.
Since the map $\alpha$ from $\tX$ to $\mathcal X_{\mathbf a,\nu}^\mathbf L$ is a principal $G_W$-bundle over a projective variety, and the norm condition defining $Y_i$ clearly cuts out a compact subset of the fibres of $\alpha$, it follows that $Y_i$ is compact, and so in particular $\beta_{|Y_i}$ is proper. Thus it remains to check that the fibres of $\beta_{|Y_i}$ are still $k$-connected.

To do this we may use the Gram-Schmidt process to iteratively deform the linear maps in the fibres in the same manner as we checked $k$-connectedness via a sequence of fibre bundles. In the case of the choice of the surjection $L_0 \to W_0$ which defines $L_0'$ and $\phi_0$, note that we need only ``Gram-Schmidt'' the frame defining the map from $K $ to $W_0/\theta(W_0)$, where $K= \text{span}_{\mathbb C}\{e_1,\ldots, e_D\}$.
\end{proof}

We can now complete the geometric proof of the equality of our inner product with that in \cite{L93}.

\begin{theorem}
The inner products $\langle, \rangle$ and $(,)$ on $\Ud$ coincide.
\end{theorem}
\begin{proof}
As discussed above, it is enough to know that the varieties $\tilde{\mathcal X}^\mathbf L_{\mathbf a+p\mathbf b_0,\nu}$ become more and more connected as $p$ tends to infinity. Since the value of $k = \mathrm{min}_{0\leq i \leq n-1}\{(a_i+p-\nu_i)\}$ clearly tends to infinity as $p$ does (where $\mathbf a$ is the representative of $\lambda \in X'$ lying in $\mathfrak S_{D,n}$) the equality of the inner products follows.
\end{proof}

\section{A construction of the canonical basis of $\dot{\mathbf U}(\widehat{\mathfrak{sl}}_n)$} 

\label{injective}
In \cite{L99a} Lusztig defines homomorphisms
\[
\psi_D\colon \mathfrak A_D \to \mathfrak A_{D-n},
\]
 which are characterised, at least on $\mathbf U_D$, by the following,
\begin{itemize}
\item $\psi_D(E_i(D))=E_i(D-n)$; \item $\psi_D(F_i(D))=F_i(D-n)$;
\item $\psi_D(K_{\mathbf a}(D))=v^{\mathbf a \cdot \mathbf
b_0}K_{\mathbf a}(D-n)$.
\end{itemize}
where $\mathbf b_0$ has all entries equal to $1$. It follows that if we work with the root datum $(X',Y')$, \textit{i.e.} with $\widehat{\mathfrak{sl}}_n$, then the maps $\psi_D$ and $\phi_D\colon \Ud\to \mathfrak A_{D,n,n}$ are compatible, that is $\psi_{D+n} \circ\phi_{D+n} = \phi_D$.

Let $\hat{\mathbf U} = \underleftarrow{\lim}_D\mathfrak A_D$ where the limit is taken over the projective system given by the maps $(\psi_D)_{D \in \mathbb N}$ described above. Since the maps $\phi_D$ are compatible with this system, there is a unique map $\phi\colon \dot{\mathbf U} \to \hat{\mathbf U}$, which factors each of the maps $\phi_D$ through the canonical map $\hat{\mathbf U} \to \mathfrak A_D$.  Theorem~\ref{same} allows us to give an alternative proof of the following injectivity result which is due to Lusztig \cite{L99a}.
 
\begin{prop}
The homomorphism $\phi$ is injective.
\label{inj}
\end{prop}
\begin{proof}
We first note that the inner product on $\dot{\mathbf U}$ is nondegenerate. While this is not explicitly stated in the book \cite{L93}, it follows easily from the results there. For example, one may use the results of $\S 26.2$ and the nondegeneracy of the inner product defined in $\S 19.1$ which is established in Lemma $19.1.4$ (all references in this sentence are to sections of \cite{L93}). Now suppose that $u$ is in the kernel of $\phi$. Then for every $D$ we have $\phi_D(u)=0$, and hence by Theorem~\ref{same} we see that $u$ is in the radical of the inner product on  $\dot{\mathbf U}$, and hence it follows that $u=0$.
\end{proof}

The modified quantum group $\dot{\mathbf U}$ is equipped with a canonical basis $\dot{\mathbf B}$ which generalizes the canonical basis of $\mathbf U^-$. We now show that the compatibility of the inner products can be used to give an essentially self-contained construction of this basis. Let $\mathbf A = \mathbb Q(v)\cap \mathbb Q[[ v^{-1}]]$ and let $\mathbb B^\pm$ be defined by:
\[
\mathbb B^\pm = \{b \in \Ud_\mathcal A: (b,b) \in 1 + v^{-1}\mathbf A, \bar{b} = b\}.
\]
We show that this set is a signed basis of $\Ud$. (We will also be able to choose a basis within this set). 

\begin{remark}
In \cite{L93} Lusztig shows that $\dot{\mathbf B}$ is ``almost orthonormal'' for the inner product, from which one can also deduce the (weaker) nondegeneracy statement used in the proof of Proposition \ref{inj}. We prefer the argument given above since we wish to give a geometric construction of the canonical basis which does not presuppose it's existence.
\end{remark}

We begin by showing that $\mathbb B^{\pm}$ is closely related to the bases $\mathfrak B_D$ of the algebras $\mathfrak A_D$.

\begin{prop}
\label{bases}
Let $b \in \mathbb B^\pm$. Then there exists $\lambda$ such that $b \in \dot{\mathbf U} 1_\lambda$. If $k$ is the residue of $\sum_{i=1}^n \lambda_i \mod n$ then there is a $p_0 >0$
such that for all $p > p_0$ we have $\phi_{k+pn}(b) \in \pm \mathfrak B_D$. Conversely, if $b \in \Ud1_\lambda$ has $\phi_D(b) \in \pm \mathfrak B_{k+pn}$ for all $p>p_1$ (some $p_1 \in \mathbb N$) then $b \in \mathbb B^{\pm}$. 
\end{prop}
\begin{proof}

Suppose that $b \in \mathbb B^{\pm}$. Then since the inner product $\dot{\mathbf U}$ is obtained as a limit from the inner products  on $\mathfrak A_D$, we see that for large $p$ we have (in the notation of Section \ref{innerU})
 \[
\sum_{l=0}^{n-1} \langle b, b\rangle_{l+pn}= 1 \text{ mod } v^{-1}\mathbb Z[v^{-1}].
 \]
 Now for each $l$, ($0 \leq l \leq n-1$) set $x_{l+pn} =\phi_{l+pn}(b)$. It is clear that $x_{l+pn}$ is bar-invariant (for the bar involution on $\mathfrak A_{l+pn}$, see Remark \ref{bartranspose}), and lies in $\mathfrak A_{l+pn,\mathcal A}$. Thus we may write $x_{l+pn} = \sum_{i \in I} a_i\{A_i\}$ for some $a_i \in \mathcal A$ and $A_i \in \mathfrak S^{\text{ap}}_{l+pn,n,n}$, where $\bar{a}_i = a_i$. Now suppose that $a_i \in v^{m}\mathbb Z[v^{-1}]$ for all $i \in I$, and $m$ is minimal with this property. Then using the ``almost orthonormality'' property that $(\{A_i\},\{A_j\}) \in  \delta_{i,j} + v^{-1}\mathbb Z[v^{-1}]$ (see Lemma \ref{almostortho}), we see that if $J\subset I$ denotes the subset consisting of those $i$ with $a_i = c_iv^m + \ldots$ where $c_i \neq 0$, then
 \[
 (x_{l+pn},x_{l+pn})_{l+pn} = (\sum_{i \in J} c_i^2)v^{2m} + \text{ lower order terms}.
 \]
 Thus in particular, since $(x_{l+pn},x_{l+pn}) \in \mathbb Z[v^{-1}]$ we must have $m=0$. But then since $x_{l+pn}$ is bar-invariant, we must have $a_i \in \mathbb Z$ for each $i \in I$. Now since 
 \[
 \sum_{k=1}^{n-1} (x_{k+pn},x_{k+pn}) \in 1+v^{-1}\mathbb Z[v^{-1}]
 \]
 it follows that in fact there is a $k \in \{0,1,\ldots, n-1\}$ such that $x_{l+pn} = 0$ for $l \neq k$ and $x_{k+pn} = \pm \{A\}$ for some $A \in \mathfrak S_{k+pn,n,n}$. Indeed the same argument shows that the signed basis $\pm\mathfrak B_D$ is characterised by the properties that its elements are bar-invariant, integral, and almost orthonormal. Note also that if $\lambda = c(A) \text{ mod } \mathbb Z\mathbf b_0$ it is then easy to see that $b = b1_\lambda$ as claimed in the statement of the Lemma.
 
The converse is easier, since we know that $\Ud$ injects into the inverse limit of the $\mathfrak A_D$, so that if $\phi_D(b) \in \mathfrak A_{D,\mathcal A}$ for all $D \equiv k \text{ mod } n$, then $b \in \Ud_\mathcal A$, and bar invariance and the condition on $(b,b)$ is also evident.
 \end{proof}

To extract a basis from $\mathbb B^{\pm}$ we need to recall some results of Lusztig. For this we need some definitions. Let $\mathfrak S_{D,n,n,}^-$ be the set of all $B \in \mathfrak S_{D,n,n}$ such that $b_{ij} = 0$ for $i >j$. Let 
$\mathfrak S_{D,n,n}^+$ be the set of all $B \in \mathfrak S_{D,n,n}$ such that $b_{ij} = 0$ for all $i <j$.
Given $A \in \mathfrak S_{D,n,n}$ we may define $A^+$ and $A^-$ in $\mathfrak S_{D,n,n}^+$ and $\mathfrak S_{D,n,n}^-$ respectively by
\[
\begin{split}
a^-_{ij} = a_{ij} \text{ if } i <j; a^-_{ij} = 0 \text{ if } i>j; a_{ii}^- = \sum_{j \in \mathbb Z, i \geq j} a_{ij}; \\
a^+_{ij} = a_{ij} \text{ if } i >j; a^+_{ij} = 0 \text{ if } i>j; a_{ii}^- = \sum_{k \in \mathbb Z, k \leq i} a_{ki};
\end{split}
\]

\begin{lemma}
\label{Lusztigresults}
Let $A \in \mathfrak S_{D,n,n}$. Then we have
\begin{enumerate}
\item If $A \in \mathfrak S_{D,n,n}^\pm$ then $\psi_D(\{A\}) = \{A-I\}$.
\item For any $A \in \mathfrak S_{D,n,n}$ we have 
\[
\{A^-\}\{A^+\} = \{A\} + \sum_{A' < A} c_{A,A'} \{A'\},
\]
where $c_{A,A'} \in \mathbb Z[v,v^{-1}]$. 

\item 
For any $A \in \mathfrak S_{D,n,n}$ we have 
\[
\psi_D(\{A\}) = \{A-I\} + \sum_{A' < A} e_{A,A'}\{A'-I\},
\]
where $\{A-I\}$, (respectively $\{A'-I\}$) are interpreted as $0$ if $A-I$ (respectively $A'-I$) does not lie in  $\mathfrak S_{D,n,n}$.

\end{enumerate}
\end{lemma}
\begin{proof}
In \cite[\S3.7]{L99a} the elements of $A \in \mathfrak S_{D,n,n}^\pm$ are related to perverse sheaves on quiver varieties attached to the cyclic quiver, giving a geometric interpretation of part of the map from $\Ud$ to $\mathfrak A_{D,n,n}$ (see also Section \ref{geometric} and \cite[\S 5]{L99} for more details). From this, and the compatibility of the maps $\phi_D$ and $\psi_D$, part $(1)$ readily follows. Part $(2)$ is \cite[Proposition 4.11]{L99}. The last part  follows by induction on the partial order $<$ using parts $(1)$ and $(2)$ together with the fact that $(A-I)^\pm = A^\pm -I$.
\end{proof}

\begin{definition}
\label{partialorder}
Next we note that the partial order $\leq$ has a combinatorial cousin $\preceq$ which we can make more explicit: Given $A,B \in \mathfrak S^{n,n}$ say $A \preceq B$ if for all $i<j \in \mathbb Z$ we have 
\[
\sum_{r \leq i; s \geq j} a_{r,s} \leq \sum_{r \leq i; s \geq j} b_{r,s},
\]
and for any $i>j$ we have
\[
\sum_{r \geq i; s \leq j} a_{r,s} \leq \sum_{r \geq i; s \leq j} b_{r,s}.
\]
It is easy to check that if $A,B \in \mathfrak S_{D,n,n}$ and $A \leq B$ then $A \preceq B$ (see for example \cite[Lemma 3.6]{BLM90} and \cite[\S 1.6]{L99}). Also, if we write $_pA = A +pI$, then it is clear that $A \preceq B$ if and only if $_pA \preceq {_pB}$. Moreover, crucially in what follows, given $A \in \mathfrak S^{n,n}$ the set 
\[
\{B \in \mathfrak S^{n,n}: B \preceq A, r(B) = r(A), c(B) = c(A)\}
\] 
is finite. 
\end{definition}

We now resolve the ambiguity of signs in the definition of $\mathbb B^{\pm}$ and extract a basis from the signed basis $\mathbb B^\pm$.

\begin{cor}
\label{signissue}
Suppose $b \in \mathbb B^{\pm}$. Then if 
\[
\mathbb B = \{b \in \mathbb B^{\pm}: \phi_D(b) \in \mathfrak B_D\cup\{0\},  \forall D \gg 0\}
\]
we have $\mathbb B^{\pm} = \mathbb B \sqcup (-\mathbb B)$. Moreover, if  $\phi_{k+pn}(b) = \{A_{k+pn}\}$ where $A_{k+pn} \in \mathfrak S_{D,n,n}$ for all $p>p_0$ say, then $A_{k+(p+1)n} = A_{k+pn} +I$.
\end{cor}
\begin{proof}
It is only necessary to show that $\mathbb B$ is well-defined. Suppose $b \in \mathbb B^{\pm}$. The previous Proposition shows that if $b = b1_\lambda$, and $k = \sum_{i=1}^n \lambda_i$, then for large enough $p$, say $p\geq p_0$, we have $\phi_{k+pn}(b) \in \pm \mathfrak B_{k+pn}$, and moreover $\phi_D(b) = 0$ if $D$ is not congruent to $k$ modulo $n$. Thus we have $\phi_{k+pn}(b) = \epsilon_{k+pn}\{A_{k+pn}\}$ where $\epsilon_{k+pn} \in \{\pm 1\}$ and $A_{k+pn} \in \mathfrak S_{k+pn,n,n}$ for all $p \geq p_0$. But now by part $(3)$ of Lemma \ref{Lusztigresults} we have
\[
\psi_{k+(p+1)n}(\{A_{k+(p+1)n}\}) = \{A_{k+(p+1)n}-I\} + \sum_{B \preceq A_{k+pn}} e_B\{B\}, \quad (e_B \in \mathcal A),
\]
whereas $\psi_{k+(p+1)n}(\phi_{k+(p+1)n}(b)) = \phi_{k+pn}(b) = \epsilon_{k+pn}\{A_{k+pn}\}$. Comparing these two expressions we can conclude that $\epsilon_{k+(p+1)n} = \epsilon_{k+pn}$ and $\{A_{k+(p+1)n}\} = \{A_{k+pn}\}$ as claimed. 
\end{proof}

\begin{cor}
\label{almostorthoUd}
The set $\mathbb B$ is almost orthonormal, that is
\[
(b_1,b_2) \in \delta_{b_1,b_2} + v^{-1}\mathbb Z[[v^{-1}]].
\]
Thus the set $\mathbb B$ is linearly independent.
\end{cor}
\begin{proof}
Let $b_1,b_2 \in \mathbb B$. If $b_1= b_11_\lambda$  then either $b_2 = b_21_\lambda$ or $b_21_\lambda = 0$ in which case the inner product $(b_1,b_2)=0$ so the claim of the Corollary holds trivially, or $b_21_\lambda = b_2$. In that case, we see from Corollary \ref{signissue} that we may find a $p_0 \in \mathbb N$ and $A,B \in \mathfrak S^{n,n}$ so that $\phi_{p_0+pn}(b) = \{_pA\}$  and $\phi_{p_0+pn}(b_2) = \{_pB\}$ for all $p \geq 0$. But then it follows from Lemma \ref{almostortho} that
\[
(\{_pA\}, \{_pB\}) \in \delta_{A,B} + v^{-1}\mathbb Z[[v^{-1}]].
\]
for all $p$, and hence taking the limit we obtain the same result for $b_1,b_2$.
To see that this implies the linear independence of the set $\mathbb B$, consider a dependence involving the minimal number of elements of $\mathbb B$:
\[
\sum_{j=1}^k p_j b_j = 0
\]
where by clearing denominators if necessary we may assume that $p_k \in \mathbb Z[v,v^{-1}]$ (and by minimality they are all nonzero) and $b_k \in \mathbb B$. We may moreover assume, multiplying through by an appropriate power of $v$, that $p_i = n_{i} + v^{-1} \mathbb Z[v^{-1}]$, where $n_{i} \in \mathbb Z$, and, reordering if necessary, that $n_{1} \neq 0$. Pick $D$ large enough so that 
\[
\sum_{l=0}^{n-1}(\phi_D(b_r), \phi_D(b_s))_{D+l} = (b_r, b_s) \mod v^{-1}\mathbb Z[[v^{-1}]], \quad 1\leq r,s \leq k.
\]
and moreover that for each $j$ we have $\phi_D(b_j) = \{B_j\}$ for some $B_j \in \mathfrak S_{D,n,n}$. Then 
\[
\begin{split}
0 = (\sum_{r=1}^k p_r b_r, \sum_{s=1}^k p_s b_s) &= \sum_{1 \leq r,s \leq k} p_rp_s(b_r,b_s) \\
& \equiv \sum_{1 \leq r,s \leq k} p_rp_s (\{B_r\}, \{B_s\})_D \mod v^{-1}\mathbb Z[v^{-1}] \\
& \equiv \sum_{1 \leq r \leq k} n_r^2 \mod v^{-1}\mathbb Z[v^{-1}].
\end{split}
\]
which is a contradiction, since $n_1 \neq 0$.
\end{proof}

We now show that if $A \in \mathfrak S^{n,n}$ then for large enough $p$ there is a unique $b \in \mathbb B$ such that $\phi_D(b) = \{_pA\}$ (where $D=\sum_{i \in [1,n],j \in \mathbb Z} a_{ij}+pn$), and hence by Corollary \ref{signissue} it will also follows that for large enough $p$ we have $\psi_D(\{_pA\}) = \{_{p-1}A\}$. We need to recall the relation between the canonical basis $\mathbf B$ of $\mathbf U^-$ and $\mathfrak B_D$. Recall from Section \ref{geometric} that the representation theory of the cyclic quiver allows us to parametrize $\mathbf B$ by tableaux $(\mu_{t,p})_{t \in \mathbb Z,p \in \mathbb N}$, where $\mu_{t,p} \in \mathbb N$ and $\mu_{t,p} = \mu_{t-n,p}$ for all $t \in \mathbb Z$ and for fixed $t$ only finitely many of the $\mu_{t.p}$ are nonzero. The $\nu$-graded part $\mathbf B_\nu$ is then indexed by $\Sigma_\nu$.

The correspondence described in Section \ref{geometric} gives a bijection between those $(\mu_{t,p})$ in $\Sigma_\nu$ satisfying
\[
\mu_{i,1}+\mu_{i,2}+\ldots \leq a_i, \quad \forall i,
\]
and the orbits in corresponding to matrices $B \in \mathfrak S_{D,n,n}^-$ with $r(B) = \mathbf a$. 
(Note that if the integers $a_i$ are sufficiently positive, this gives an injection from $\mathbf B_\nu$ into $\mathfrak S_{D,n,n}^-$.) Using the same correspondence, composed with the transpose map $\Psi$, we obtain a similar correspondence between (appropriate subsets of) $\mathbf B_\nu$ and elements of $\mathfrak S_{D,n,n}^+$. These can be combined to give a correspondence between the set of triples $\mathcal T = \{(b_1,b_2,\lambda): b_1,b_2 \in \mathbf B, \lambda \in X\}$ and elements of $\mathfrak S^{n,n}$ as follows: the elements $b_1,b_2$ corresponds to tableau $(\mu_{t,p})$ and $(\rho_{t,p})$ say, and we define $A = A(b_1,b_2,\lambda) \in \mathcal S^{n,n}$ by setting:
\[
a_{ij} = \left\{\begin{array}{cc}\mu_{i,j-i}, & \text{ if } i<j, \\\rho_{j,i-j} & \text{ if } i>j; \\ \lambda_i - \sum_{t\geq 1} \mu_{i,t} -\sum_{s \geq 1} \rho_i & \text{ if } i=j.\end{array}\right.
\]
We will write $b_A$ for the element $b_1^+1_\lambda b_2^- \in \Ud$, and $_Db_A$ for its image under $\phi_D$. 
It follows from the \cite[\S 5]{L99} and \cite[Proposition 4.11]{L99} that if $\sum_{i=1}^n \lambda_i = D$ and the entries of $A(b_1,b_2,\lambda)$ are all non-negative, then 
\begin{equation}
\label{uptriangle}
_Db_A = \{_pA\} + \sum_{B \prec {_pA}} e_{B,{_pA}}\{B\}
\end{equation}

\begin{prop}
\label{stabilization}
Let $A \in \mathfrak S^{n,n}$. For large enough $p$ we have 
\[
\psi_D(\{_pA\}) = \{_{p-1}A\}
\]
\end{prop}
\begin{proof}
Via the bijection described above between $\mathcal T$ and $\mathfrak S^{n,n}$, we may find a subset $\mathcal T_\mathbf a$ of $\mathcal T$ such that $\{A(b_1,b_2,\mathbf a)\}$  is a basis of $\mathfrak A_{D}[\mathbf i_\mathbf a]$ as $(b_1,b_2,\mathbf a)$ runs over the set $\mathcal T_\mathbf a$. Then the elements $_Db_A$ are clearly also a basis of $\mathfrak A_D[\mathbf i_\mathbf a]$ since they are related to the elements $\{A\}$ by an upper triangular matrix, and moreover they satisfy:
\begin{enumerate}
\item 
$\bar{_Db_A} = {_Db_A}$,
\item 
$_Db_A \in \mathbf U_{D,\mathcal A}$.
\end{enumerate}
As in the proof of Proposition \ref{bases}, the basis $\{ \{A\}: A \in \mathfrak S_{D,n,n}\}$ is characterized up to sign by the properties of being bar-invariant, integral (that is, contained in $\mathbf U_{D,\mathcal A}$), and being almost orthonormal, so that
\[
(\{A\},\{B\})_D \in \delta_{A,B} + v^{-1}\mathbb Z[v^{-1}].
\]
(In fact, Proposition \ref{bases} shows that less than this characterises $\pm \mathfrak B_D$). 

We  now show that one can obtain $\{A\}$ from $_Db_A$ by a Gram-Schmidt style process. Indeed if $A$ is minimal for the ordering $\preceq$, then clearly $\{A\} = b_A$. Thus we consider the following claim:
\begin{itemize}
\item
For each $A \in \mathfrak S^{n,n}$, there is a $p_0 \in \mathbb Z$ such that for all $p>p_0$ we have 
\[
\{_pA\} = {_Db_A} + \sum_{A'<A} d_{A',A}{_Db_{A'}}.
\]
where $d_{A',A} \in \mathcal A$ do not depend on $p$, and $D = pn+ \sum_{i,j: 1\leq i \leq n} a_{ij}$.
\end{itemize}

Note that the proof of the proposition now follows immediately since $\psi_D(_Db_A) = _{D-n}b_A$. We show this by induction on $\preceq$, since if $A$ is minimal, then Equation (\ref{uptriangle}) we see that $b_A = \{A\}$, and we are done. Thus suppose that the result is known for all $B\prec A$, and let $I$ be the (finite) set 
\[
\{B \in \mathfrak S^{n,n}:  B\preceq A, r(A) = r(B), c(A)= c(B)\}.
\]
Now for $x$ in the span of $\{\{B\}: B \in I\}$, set $N(x) = \text{max}\{\nu(x,\{B\})_D: B \in I, B \neq A\}$, where for $f \in \mathcal A$ we let $\nu(f)$ denote the highest power of $v$ occuring in $f$. Let $N = N(b_A)$, and suppose that $N \geq 0$. Let $J$ denotes the subset of $I$ for which $\nu(b_A,\{B\})=N$, so that if $B \in J$ we have 
\[
(b_A,\{B\})_D = c_Bv^N+ \text{ lower order terms}, \quad (c_B \in \mathbb Z).
\]
Now $(\{B\},\{B\})_D \in 1 + v^{-1}\mathbb Z[v^{-1}]$, so we may recursively solve for $a_B \in v^{-N}\mathbb Z[v]\cap \mathbb Z[v^{-1}]$ such that $a_B.(\{B\},\{B\})_D \in 1+ v^{-N-1}\mathbb Z[v^{-1}]$. It follows immediately that we may find $e_B \in \mathcal A$ such that $\bar{e}_B = e_B$ and $e_B = c_Bv^Na_B \text{ mod } v^{-1}\mathbb Z[v^{-1}]$. Then we set
\[
b'_A = b_A -\sum_{B \in J}e_B\{B\}.
\]
It follows from the almost orthonormality of the $\{B\}$ that $(b'_A,\{B\})_D \in v^{N-1}\mathbb Z[v^{-1}]$, and $b'_A$ is again bar-invariant, lies in $\mathbf U_{D,\mathcal A}$, and satisfies $N(b'_A)<N$. We may thus iterate this construction to obtain an element $b''_A$ which has $N(b''_A) \leq -1$, is bar-invariant, and lies in $\mathbf U_{D,\mathcal A}$. But then we claim that we must have $b''_A = \{A\}$. Indeed we know (using Equation (\ref{uptriangle})) that we can write 
\[
b''_A= \{A\} + \sum_{B \prec A} f_B\{B\}
\]
for some $f_B \in \mathcal A$ with $\bar{f}_B = f_B$. If it is not the case that $f_B =0$ for all $B$, then there is some $B$ with $\nu(f_B) \geq 0$ maximal, whence we see that $\nu(b''_A,\{B\})_D \geq 0$ which is a contradiction. Thus $b''_A = \{A\}$ as required.

Now examining the above process, we see that it uses only the values of $(b_A,\{B\})_D$ down to order to $v^{-N(b_A)}\mathbb Z[v^{-1}]$, and by induction we see that these, for large enough $p$ are determined by the values of $(b_A,b_B)_D$, down to some possibly lower order (determined by the coefficients $d_{B,C}$). Since the values of $(b_A,b_B)_D$ converge in $\mathbb Z((v^{-1}))$ we see that we may find a large enough $p_0$ so that $\{_pA\}$ is a linear combination of $\{\phi_D(b_{A'}): A' \preceq A\}$ with coefficients independent of $p$ as required.
\end{proof}

We can now show that the set $\mathbb B$ is a basis of $\dot{\mathbf U}$.

\begin{theorem}
$\mathbb B$ is a basis of $\Ud$. 
\end{theorem}
\begin{proof}
Notice first that given any $b \in \mathbb B$, Proposition \ref{bases} implies that $\phi_D(b) \in \mathfrak B_D\cup \{0\}$ for large enough $D$, and is nonzero provided $D$ has a fixed residue modulo $n$. 
By Corollary \ref{almostorthoUd} we know that the elements of $\mathbb B$ are linearly independent, so we need only show that they span $\Ud$. To do this it is enough to show that the element $b_1^+b_2^-1_\lambda$ for $b_1,b_2 \in \mathbf B$ and $\lambda \in X$ lie in the span of $\mathbb B$, since they form a basis for $\Ud$. But the claim in the proof of Proposition \ref{stabilization} shows that we may find an element of $\mathbb B$ which is a linear combination of such basis elements with leading coefficent $1$, so that the matrix relating the two sets is invertible and $\mathbb B$ indeed spans $\Ud$.

\end{proof}

\begin{remark}
\label{SVremark}
The results of \cite[\S 26.3]{L93} then show that $\mathbb B = \dot{\mathbf B}$, and thus the results of this section give a new proof of the conjecture made in \cite[\S 9.3]{L99}, which was originally proved in \cite{ScV}. Our goal here was to construct the canonical basis purely within the context of the inverse system $\mathbf U_D$, thus unlike \cite{ScV}, we do not need to assume the existence of $\dot{\mathbf B}$, nor use any properties of crystal bases. It should be noted however that by using results of Kashiwara on global crystal bases, \cite{ScV} have obtained a more precise result (also conjectured by Lusztig \cite{L99a}) which shows that the maps $\phi_D$ are all compatible with the canonical basis, \textit{i.e.} if $b \in \dot{\mathbf B}$ then $\phi_D(b) \in \mathfrak B_D \cup \{0\}$, and moreover the kernel of $\phi_D$ is spanned by a subset of $\dot{\mathbf B}$. The results of this section show that this theorem would also follow if we could show that the maps  $\psi_D$ are compatible with the bases $\mathbf B_D$ and $\mathbf B_{D-n}$, a question which can be phrased purely geometrically (in terms of perverse sheaves). Note that it is \textit{not} true that the maps $\psi_D$ send $\mathfrak B_D$ to $\mathfrak B_{D-n} \cup \{0\}$, as was  pointed out already in \cite[1.12]{L99a}. It is possible to give a construction of the maps $\psi_D$ in the context of perverse sheaves on the ind-varieties $\mathcal F_{\mathbf a}$, (\textit{i.e.} to show that their exists a functor on the derived category which preserves perverse sheaves (up to shift) and induces $\psi_D$ on the Grothendieck group which moreover is compatible with the ''convolution'' on $\mathcal K_{D,n,n}$) but it is not immediately clear why this functor preserves simple objects.
\end{remark}

\section{A Postitivity Result}

We may combine Theorem~\ref{same} and Proposition~\ref{same2} to prove a positivity result for the inner product of two elements of $\dot{\mathbf B}$. This has been conjectured by Lusztig for all types.

\begin{theorem}
Let $b_1, b_2 \in \dot{\mathbf B}$.
\[
(b_1, b_2) \in \mathbb N[[v^{-1}]]\cap\mathbb Q(v).
\]
\label{pos}
\end{theorem}
\begin{proof}
We may assume that there is a $\lambda \in X$ such that $b_11_\lambda=
b_1$, and $b_21_\lambda= b_2$.
Let $k \in \{0,1,\dotsc,n-1\}$ be such that
$\sum_{j=1}^n\lambda_j=k$ mod $n$. Then
\[
(b_1, b_2)=\lim_{p \to \infty}(\phi_{k+pn}(b_1),\phi_{k+pn}(b_2))_{k+pn}
\]
By Proposition \ref{bases} we know that for all large enough $D$ we have $\phi_D(b_1), \phi_D(b_2)$
are in $\mathfrak B_D$, hence it is clear from Equation (\ref{perversedef}) that
\[
(\phi_{k+pn}(b_1),\phi_{k+pn}(b_2))_{k+pn} \in
\mathbb N[v,v^{-1}].
\]
However, it follows also from Lemma~\ref{values2} that the
left-hand side is in fact in $\mathbb N[v^{-1}]$ (this can also be seen directly,
using the definition of intersection cohomology sheaves). Hence $(b_1, b_2)$ is the limit of
elements of $\mathbb N[v^{-1}]$, and the statement follows.
\end{proof}

\begin{remark}
All the results of this paper have analogues for the nonaffine case, which can be proved in exactly the same way. The module $V$ is replaced by a $D$-dimensional vector space over $\mathbf k$, and the space $\mathcal F^n$ of n-step periodic lattices should be replaced by the space of n-step flags in that vector space. In this case  the algebra corresponding to $\mathbf U_D$ is actually equal to the algebra analogous to $\mathfrak A_D$, hence the results are sometimes more straightforward.
\end{remark}

\noindent
\textit{Acknowledgements}: This paper is based on a chapter of my thesis which was written under the direction of George Lusztig. I would like to thank him both for posing the problem that led to this paper, and for our many conversations about the contents of this paper and much else besides. I would also like to thank Jared Tanner for a useful conversation.

\end{document}